\documentclass{article}
\usepackage{geometry}
\usepackage{graphics}
\usepackage{epstopdf}
\usepackage{color}
\usepackage{amsmath,amsfonts}
 \usepackage[pagewise]{lineno}
 \usepackage{tikz}
\usepackage{graphicx}
\usepackage{tikz-3dplot}
\usepackage{setspace}
\usepackage{fancyhdr}
\usepackage[ansinew]{inputenc}
\usepackage{comment}
\usepackage{amssymb}
\usepackage{latexsym}
\usepackage{graphics}
\usepackage{epstopdf}
\usepackage{xcolor}
\usepackage{fancybox}
\usepackage{amsmath,amsfonts}
\usepackage{hyperref}
\begin{document}

\setlength{\parskip}{0.3\baselineskip}

\newtheorem{theorem}{Theorem}
\newtheorem{corollary}[theorem]{Corollary}
\newtheorem{lemma}[theorem]{Lemma}
\newtheorem{proposition}[theorem]{Proposition}
\newtheorem{definition}[theorem]{Definition}
\newtheorem{remark}[theorem]{Remark}
\renewcommand{\thefootnote}{\alph{footnote}}
\newenvironment{proof}{\smallskip \noindent{\bf Proof}: }{\hfill $\Box$\hspace{1in} \medskip \\ }
%%%%%%%%%%%%%%%

\newcommand{\sii}{\Leftrightarrow}
\newcommand{\imer}{\hookrightarrow}
\newcommand{\imerc}{\stackrel{c}{\hookrightarrow}}
\newcommand{\Con}{\longrightarrow}
\newcommand{\con}{\rightarrow}
\newcommand{\conf}{\rightharpoonup}
\newcommand{\confe}{\stackrel{*}{\rightharpoonup}}
\newcommand{\pbrack}[1]{\left( {#1} \right)}
\newcommand{\sbrack}[1]{\left[ {#1} \right]}
\newcommand{\key}[1]{\left\{ {#1} \right\}}
\newcommand{\dual}[2]{\langle{#1},{#2}\rangle}

%%%%%%%%%%%%%%%%%%
\newcommand{\R}{{\mathbb R}}
\newcommand{\N}{{\mathbb N}}
\newcommand{\cred}[1]{\textcolor{red}{#1}}
%%%%%%%%%%%%%%%%

\title{\bf Stability and Regularity for Double Wall Carbon Nanotubes Modeled as Timoshenko Beams with Thermoelastic Effects and Intermediate Damping}
\author{Fredy M.  Sobrado  Su\'arez$^*$\\
{\small Department of  Mathematics, The  Federal University of Technological of Paran\'a, Brazil}\\Lesly D.  Barbosa  Sobrado,   Gabriel L. Lacerda de Araujo\quad \\
{\small Institute of Mathematics,  Federal University   of  Rio de Janeiro, Brazil}\\
and \quad
Filomena B.  Rodrigues Mendes\\
{\small Department of Engenhary Electric, The  Federal University of Technological of Paran\'a, Brazil}
}
\date{}
\maketitle

\let\thefootnote\relax\footnote{$^*$ corresponding author:{\it e-mail:}   {\rm fredy@utfpr.edu.br} (Fredy M. Sobrado  Su\'arez)}.

%%%%%%%%%%%%%%%%%%%%%%%%
\begin{abstract}
This research studies two systems composed by the Timoshenko beam model for double wall carbon nanotubes, coupled with the heat equation governed by Fourier's law. For the first system, the coupling is given by the speed the rotation of the vertical filament in the beam $\beta\psi_t$ from the first beam of Tymoshenko and the Laplacian of temperature $\delta\theta_{xx}$, where we also consider the damping terms fractionals $\gamma_1(-\partial_{xx})^{\tau_1}\phi_t$, $\gamma_2(-\partial_{xx})^{\tau_2} y_t$ and $\gamma_3(-\partial_{xx})^{\tau_3} z_t$, where $(\tau_1, \tau_2, \tau_3) \in [0,1]^3$. For this first system we proved that the semigroup $S_1(t)$ associated to system decays exponentially for all $(\tau_1 , \tau_2 , \tau_3 ) \in [0,1]^3$. The second system also has three fractional damping  $\gamma_1(-\partial_{xx})^{\beta_1}\phi_t$, $\gamma_2(-\partial_{xx})^{\beta_2} y_t$ and  $\gamma_3(-\partial_{xx})^{\beta_3} z_t$, with $(\beta_1, \beta_2, \beta_3) \in [0,1]^3$. Furthermore, the couplings between the heat equation and the Timoshenko beams of the double wall carbon nanotubes for the second system is given by the Laplacian of the rotation speed of the vertical filament in the beam $\beta\psi_{xxt}$ of the first beam of Timoshenko and the Lapacian of the temperature $\delta\theta_{xx}$. For the second system, we prove the exponential decay of the associated semigroup $S_2(t)$ for  $(\beta_1, \beta_2, \beta_3) \in [0,1]^3$ and also show  that this semigroup admits Gevrey classes $s>(\phi+1)/(2\phi)$ for $\phi=\min\{\beta_1,\beta_2,\beta_3\}, \forall (\beta_1,\beta_2,\beta_3)\in (0,1)^3$,  and we finish our investigation proving that $S_2(t)$ is analytic when the parameters $(\beta_1, \beta_2, \beta_3) \in [1/2,1]^3$. One of the motivations for this research was the work recently published in 2023; Ramos et al. \cite{Ramos2023CNTs}, whose partial results are part of our results obtained for the first system for $(\tau_1, \tau_2, \tau_3) = (0, 0, 0)$.

\end{abstract}

\bigskip
{\sc keyword:} Asymptotic Behavior,  Stability, Regularity,   Analyticity,  DWCNTs-Fourier System, Gevrey Class.

%%%%%%%%%%%%%%%%%%%%%%%%%
\setcounter{equation}{0}

\section{Introduction}

The discovery  of structures called carbon nanotubes (CNTs) occurred in 1987 and later officially disclosed to the scientific community in 1991 \cite{SL1991} as the multi wall carbon nanotubes (MWCNTs); they were discovered experimentally in the search for a molecular structure called Fullerene. Fullerene is a closed carbon structure with a spherical format (geodesic dome) formed by 12 pentagons and 20 hexagons, whose formula is $C_{60}$. Carbon nanotubes are cylindrical macromolecules composed of carbon atoms in a periodic hexagonal array with $sp^2$ hybridization, similar to graphite \cite{HDTR08}. They are made like rolled sheets of graphene and can be as thick as a single carbon atom. They receive this name due to their tubular morphology in nanometric dimensions ($1nm=10^{-9}m.$). According to Shen and Brozena \cite{SBW2011}, CNTs are classified in three ways: single wall carbon nanotubes (SWCNTs), double wall carbon nanotubes (DWCNTs) and  (MWCNTs), where the concentric cylinders interact with each other through the Van der Walls force, the authors also point out that DWCNTs are an emerging class of carbon nanostructures and represent the simplest way to study the physical effects of coupling between the walls of carbon nanotubes.
%%%%%%%%%%%%%%%%%%%%%%%%%%%%%%%
%\hspace*{-4.0cm}
\begin{figure}[ht]
\center
\begin{minipage}{0.4\textwidth}
\center
\includegraphics[width=0.4\textwidth]{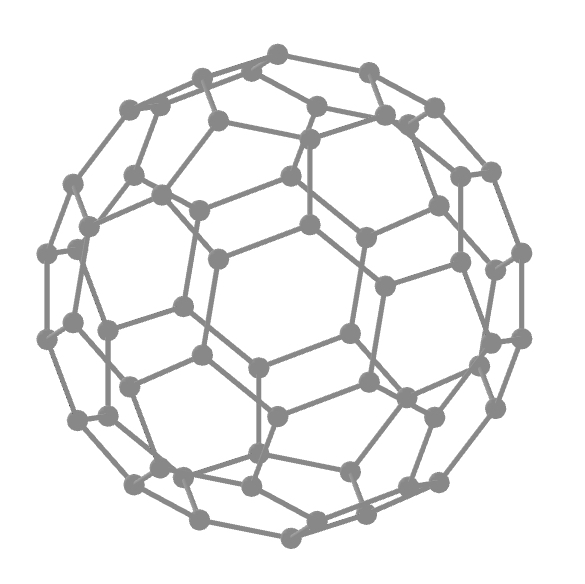}
\caption{Fullerene $C_{60}$}
\label{Left}
\end{minipage}\hspace*{1cm}
\begin{minipage}{0.4\textwidth}
\center
\includegraphics[width=0.4\textwidth]{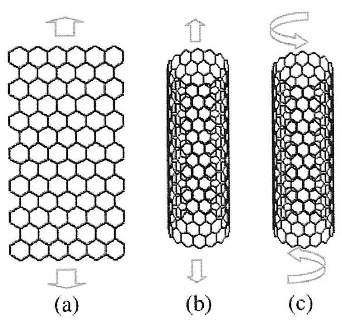}
\caption{Nanotube structure \cite{Silva2009}}
\label{Right}
\end{minipage}
\end{figure}

The discovery of this new structure at the molecular level contributed in the last decade to the advancement of nanotechnology. In \cite{YM2003} an analysis of the main properties of CNTs was presented, the study confirmed that CNTs have excellent properties; mechanical, electronic and chemical; they are about ten times stronger and six times lighter than steel. They transmit electricity like a superconductor and are excellent transmitters of temperature. Due to their superior electronic and mechanical properties to currently used materials, carbon nanotubes are candidates to be used in products and equipment that require nanoscale structures.

 In the future, CNTs should become the base material for nanoelectronics, nanodevices, and nanocomposites. The main problems that have to be overcome for this to happen are the difficult controlled experiments at the nanoscale: the high cost of molecular dynamics simulations and the high time consumption of these simulations. Knowing better the models of continuous mechanics, which are governed by the modeling through the Euler elastic beam model and the Timoshenko beam model  used to study the mechanics of linear and nonlinear deformations, should help to make this possible.
 
 The Euler-Bernoulli beam model disregards the effects of shear and rotation, and according to \cite{YM2002, YM2003} the vibrations in carbon nanotubes are animated by high frequencies, above $1Thz$. According to Yoon and others \cite {YM2005}, the effects of rotational inertia and shear are significant in the study of terahertz frequencies ($10^{12}$), hence Yoon \cite{YM2005} considers questionable the Euler-Bernoulli model applied to (CNTs). Therefore, the Timoshenko Model is the most suitable. For double-walled nanotubes (DWCNTs) or concentric multi-walled nanotubes (MWCNTs), the most used continuous models in the literature assume that all nested tubes of MWCNTs remain coaxial during deformation and thus can be described by a single deflection model. However, this model cannot be used to describe the relative vibration between adjacent tubes of MWCNTs. In 2003,  it was proposed by \cite{YM2003} that the fittings of concentric tubes are considered as individual beams, and that the deflections of all nested tubes are coupled through the van der Waals interaction force between two adjacent tubes \cite{JC2007,JC2008}. So, each of the inner and outer tubes is modeled as a beam.
 
 In the pioneering work on the carbon nanotube model by Yoon et al. \cite{YM2004}, the authors proposed a coupled system of partial differential equations inspired by the Timoshenko beam model to model DWCNTs. The model consists of the following equations
 \begin{eqnarray*}
 \rho A_1\dfrac{\partial^2 Y_1}{\partial t^2}-\kappa G A_1\bigg(\dfrac{\partial^2 Y_1}{\partial x^2}-\dfrac{\partial \varphi_1}{\partial x}   \bigg)-P & = & 0,\\
 \rho I_1 \dfrac{\partial^2 \varphi_1}{\partial t^2}-E I_1 \dfrac{\partial^2 \varphi_1}{\partial x^2}-\kappa G A_1 \bigg( \dfrac{\partial Y_1}{\partial x}-\varphi_1 \bigg ) & = & 0,\\
 \rho A_2\dfrac{\partial^2 Y_2}{\partial t^2}-\kappa G A_2\bigg(\dfrac{\partial^2 Y_2}{\partial x^2}-\dfrac{\partial \varphi_2}{\partial x}   \bigg)+P & = & 0,\\
 \rho I_2 \dfrac{\partial^2 \varphi_2}{\partial t^2}-E I_2 \dfrac{\partial^2 \varphi_2}{\partial x^2}-\kappa G A_2 \bigg( \dfrac{\partial Y_2}{\partial x}-\varphi_2 \bigg ) & = & 0.
 \end{eqnarray*}
 Where $Y_i$ and $\varphi_i$ ($i=1,2$) represent respectively the total deflection and the inclination due to the bending of the nanotube $i$ and the constants $I_i$, $A_i$ denote the moment of inertia and the cross-sectional area of the tube $i$,  respectively, and $P$ is the Van der Waals force acting on the interaction between the two tubes per unit of axial length. Also according to \cite{YM2004}, it can be seen that the deflections of the two tubes are coupled through the Van der Waals interaction $P$ (see \cite{Timoshenko1921}) between the two tubes, and as the tubes inside and outside of a DWCNTs are originally concentric, the Van der Waals interaction is determined by the spacing between the layers. Therefore, for a small-amplitude linear vibration, the interaction pressure at any point between the two tubes linearly depends on the difference in their deflection curves at that point, that is, it depends on the term
 \begin{equation}\label{Eq14ARamos}
 P=\jmath (Y_2-Y_1).
 \end{equation}
 In  particular, the Van der Waals interaction coefficient $\jmath$ for the interaction pressure per unit axial length can be estimated based on an effective interaction width of the tubes as found in \cite{YM2003A,Ru2000}. Thus, this model treats each of the nested and concentric nanotubes as individual Timoshenko beams interacting in the presence of Van der Waals forces (see Figure: \eqref{DWCNTs}).
 \begin{figure}[ht]
\begin{minipage}{1\textwidth}
\center
\includegraphics[width=0.95\textwidth]{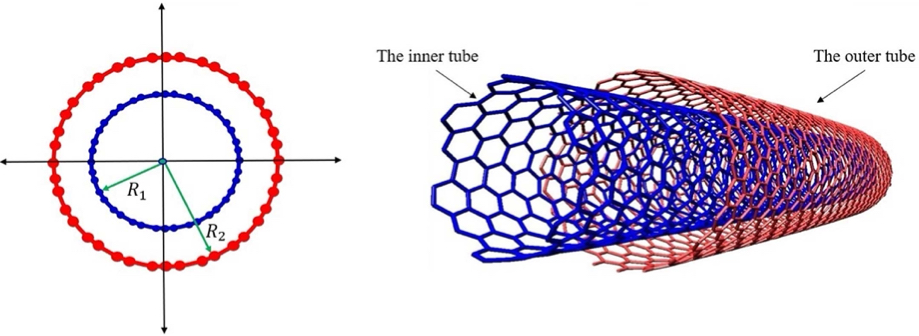}
\caption{2D and 3D Representations of the Double Wall Carbono Nanotubes Model \cite{Ramos2023CNTs}}
\label{DWCNTs}
\end{minipage}
\end{figure}
 
Currently in the literature there are few investigations related to the study of asymptotic behavior and/or regularity for DWCNTs models, or for DWCNTs systems coupled with the heat equation governed by Fourier's law (DWCNTs-Fourier). The DWCNTs model   was studied in 2015 in the thesis \cite{Nunes2015}, where the author studied the asymptotic behavior of the model:
\begin{eqnarray}
\label{Eq0.1}
\rho_1 \varphi_{tt}-\kappa_1(\varphi_x-\psi)_x-\jmath(y-\varphi)+\alpha_0\varphi_t=0\quad{\rm in} \quad (0,l)\times (0,\infty),\\
\label{Eq0.2}
\rho_2\psi_{tt}-b_1\psi_{xx}-\kappa_1(\varphi_x-\psi)+\alpha_1\psi_t=0\quad{\rm in} \quad (0,l)\times (0,\infty),\\
\label{Eq0.3}
\rho_3y_{tt}-\kappa_2(y_x-z)_x+\jmath(y-\varphi)+\alpha_2y_t=0\quad{\rm in} \quad (0,l)\times (0,\infty),\\
\label{Eq0.4}
\rho_4z_{tt}-b_2z_{xx}-\kappa_2(y_x-z)+\alpha_3 z_t=0\quad{\rm in} \quad (0,l)\times (0,\infty),
\end{eqnarray}
with the initial conditions
\begin{eqnarray}
\label{Eq01.5}
\varphi(x,0)=\varphi_0(x),\quad \varphi_t(x,0)=\varphi_1(x),\quad \psi(x,0)=\psi_0(x),\quad \psi_t(x,0)=\phi_1(x) & {\rm in}\quad x\in (0,l),
\\
\label{Eq01.6}
y(x,0)=y_0(x),\quad y_t(x,0)=y_1(x),\quad z(x,0)=z_0(x),\quad z_t(x,0)=z_1(x) & {\rm in}\quad x\in (0,l),
\end{eqnarray}
and subject to boundary conditions
\begin{eqnarray}
\label{Eq0.5}
\varphi(0,t)=\varphi(l,t)=\psi(0,t)=\psi(l,t)=0\quad{\rm for\quad all}\quad t>0,\\
\label{Eq0.6}
y(0,t)=y(l,t)=z(0,t)=z(l,t)=0\quad{\rm for\quad all}\quad t>0.
\end{eqnarray}
For the case that $\alpha_0=0$ and $\alpha_i>0$, for $i=1, 2, 3$, in \cite{Nunes2015} the author demonstrated the lack of exponential decay of the semigroup $(S(t))_{t\geq 0}$ associated with the system \eqref{Eq0.1}--\eqref{Eq0.6}, when $\frac{\rho_1}{\kappa_1}\not =\frac{\rho_2}{b_1}$ and $\jmath\big(\frac{\rho_2}{b_1}-\frac{\rho_1}{\kappa_1}\big)\not =\frac{\kappa_1}{ b_1}$, and also proved that if $\chi= \frac{\kappa_1\rho_2-b_1\rho_1}{\kappa_1^2-\jmath \rho_2\kappa_1+\jmath b_1\rho_1}=0$, then $(S(t))_{t\geq 0}$ is exponentially stable,  and if $\chi\not=0$,  $(S(t))_{t\geq 0}$ is exponentially stable. Beyond, is proved that if $\chi\not=0$, then $(S(t))_{t\geq 0}$ is polynomially stable with optimal rate $o(t^{-\frac{1}{2}})$. In addition, in Chapter 4 of \cite{Nunes2015}, the author validates through numerical analysis, using the finite difference method, the results previously demonstrated, and in addition to presenting graphs of other cases, such as considering $\alpha_i=0$ and $\alpha_i>0$ for $i=1,2,3,4$.

 Recently, in 2023, two new investigations emerged: One of them is the DWCNTs-Fourier system with friction dampers, see  \cite{Ramos2023CNTs}; in this work 
the authors consider the problem of heat conduction in carbon nanotubes modeled as Timoshenko beams, inspired by the work of Yoon et al. [Comp. Part B: Ing. 35 (2004) 87--93]. The system is given by
\begin{eqnarray}
\label{Eq0.10}
\rho_1 \varphi_{tt}-\kappa_1(\varphi_x-\psi)_x-\jmath(y-\varphi)+\gamma_1\varphi_t=0\quad{\rm in} \quad (0,l)\times (0,\infty),\\
\label{Eq0.11}
\rho_2\psi_{tt}-b_1\psi_{xx}-\kappa_1(\varphi_x-\psi)+\delta\theta_{xx}=0\quad{\rm in} \quad (0,l)\times (0,\infty),\\
\label{Eq0.12}
\rho_3y_{tt}-\kappa_2(y_x-z)_x+\jmath(y-\varphi)+\gamma_2y_t=0\quad{\rm in} \quad (0,l)\times (0,\infty),\\
\label{Eq0.13}
\rho_4z_{tt}-b_2z_{xx}-\kappa_2(y_x-z)+\gamma_3 z_t=0\quad{\rm in} \quad (0,l)\times (0,\infty),\\
\label{Eq0.14}
\rho_5\theta_t-K\theta_{xx}+\beta\psi_t=0\quad {\rm in}\quad (0,l)\times (0,\infty),
\end{eqnarray}
subject to boundary conditions \eqref{Eq0.5}, \eqref{Eq0.6} and 
\begin{equation}
\label{Eq0.15}
\theta(0,t)=\theta(l,t)=0\qquad \text{ for all}\quad t>0.
\end{equation}
Note that the system \eqref{Eq0.5}--\eqref{Eq0.15} presents three friction dissipators (weak damping): $\gamma_1\varphi_ t,\gamma_2y_t$ and $\gamma_3 z_t$. The authors apply semigroup theory of linear operators to demonstrate the exponential stabilization of the semigroup  $S(t)$ associated with the system \eqref{Eq0.5}--\eqref{Eq0.15}, and their results are independent of the relationship between the coefficients. Furthermore, they analyze the totally discrete problem using a finite difference scheme, introduced by a space-time discretization that combines explicit and implicit integration methods. The authors also show the  construction of numerical energy and simulations that validate the theoretical results of exponential decay and convergence rates.

 By the year of 2023,   \cite{MDCL2023} investigated the one-dimensional equations for the double wall carbon nanotubes modeled by coupled Timoshenko elastic beam system with nonlinear arbitrary localized damping:
 \begin{eqnarray}
\label{Eq0.20}
\rho_1 \varphi_{tt}-\kappa_1(\varphi_x-\psi)_x-\jmath(y-\varphi)+\alpha_1(x)g_1(\varphi_t)=0\quad{\rm in} \quad (0,l)\times (0,\infty),\\
\label{Eq0.21}
\rho_2\psi_{tt}-b_1\psi_{xx}-\kappa_1(\varphi_x-\psi)+\alpha_2(x)g_2(\psi_t)=0\quad{\rm in} \quad (0,l)\times (0,\infty),\\
\label{Eq0.22}
\rho_3y_{tt}-\kappa_2(y_x-z)_x+\jmath(y-\varphi)+\alpha_3(x)g_3(y_t)=0\quad{\rm in} \quad (0,l)\times (0,\infty),\\
\label{Eq0.23}
\rho_4z_{tt}-b_2z_{xx}-\kappa_2(y_x-z)+\alpha_4(x)g_4 (z_t)=0\quad{\rm in} \quad (0,l)\times (0,\infty),
\end{eqnarray}
where the localizing functions $\alpha_i(x)$ are supposed to be smooth and nonnegative, while the nonli-near functions $g_i(x), i = 1,\cdots, 4$, are continuous and monotonic increasing. The system \eqref{Eq0.20}--\eqref{Eq0.23} is subject to Dirichlet boundary conditions; in \eqref{Eq0.5} and \eqref{Eq0.6}, see \cite{MDCL2023},  the authors showed that damping placed on an arbitrary small support, not quantized at the origin, leads to uniform (time asymptotic) decay rates for the energy function of the system.

In the same direction of this last paper, we would like to mention the work of Shubov and Rojas-Arenaza \cite{SR2010} where they considered the system \eqref{Eq0.20}-\eqref{Eq0.23}  with $\alpha_i(x)=1, g_i(s)=s, i=1,\cdots,4$, initial conditions \eqref{Eq01.5}-\eqref{Eq01.6}. Subject to boundary conditions of type:
\begin{equation}\label{Eq2.1MG}
\left\{\begin{array}{cc}
\kappa_1(\varphi_x-\psi)(l,t)=-\rho_2\gamma_1\varphi_t(l,t) & t\geq 0,\\
b_1\psi_x(l,t)=-\rho_2\gamma_2\psi_t(l,t),  & t\geq 0,\\
\kappa_2(y_x-z)(l,t)=\rho_4\gamma_3y_t(l,t), & t\geq 0,
\\
b_2z_x(l,t)=-\rho_4\gamma_4z_t(l,t), & t\geq 0.
\end{array}\right.
\end{equation}
They first proved that the energy associated to the system,  with boundary conditions \eqref{Eq2.1MG}, is decreasing if $\jmath = 0$. After they proved that the semigroup generator is an unbounded non self-adjoint operator with a compact resolvent.
 
The two systems that we study in this research are for models of carbon nanotubes coupled with the heat equation given by Fourier's Law. The difference between these two systems is in the coupling of the DWCNTs model and the heat equation. The first system is a generalization of the model presented in \cite{Ramos2023CNTs}: we consider the 3 fractional damping; $\gamma_1(-\partial_{xx})^{\tau_1}\varphi_t$, $\gamma_2(- \partial_{xx})^{\tau_2}y_t$ and $\gamma_3(-\partial_{xx})^{\tau_3}z_t$, for the parameters $\tau_i, i=1,2,3$, varying in the interval $[0,1]$. We note that when $(\tau_1,\tau_2,\tau_3)=(0,0,0)$ the system is the one studied in \cite{Ramos2023CNTs}. The second system studied in this work is given by:
\begin{eqnarray}
\label{Eq2.1}
\rho_1 \varphi_{tt}-\kappa_1(\varphi_x-\psi)_x-\jmath(y-\varphi)+\gamma_1(-\partial_{xx})^{\beta_1}\varphi_t=0\quad{\rm in} \quad (0,l)\times (0,\infty),\\
\label{Eq2.2}
\rho_2\psi_{tt}-b_1\psi_{xx}-\kappa_1(\varphi_x-\psi)+\delta \theta_{xx}=0\quad{\rm in} \quad (0,l)\times (0,\infty),\\
\label{Eq2.3}
\rho_3y_{tt}-\kappa_2(y_x-z)_x+\jmath(y-\varphi)+\gamma_2(-\partial_{xx})^{\beta_2}y_t=0\quad{\rm in} \quad (0,l)\times (0,\infty),\\
\label{Eq2.4}
\rho_4z_{tt}-b_2z_{xx}-\kappa_2(y_x-z)+\gamma_3(-\partial_{xx})^{\beta_3}z_t=0\quad{\rm in} \quad (0,l)\times (0,\infty),\\
\label{Eq2.5}
\rho_5\theta_t-K\theta_{xx}-\delta \psi_{xxt}=0\quad{\rm in} \quad (0,l)\times (0,\infty).
\end{eqnarray}
We study the system \eqref{Eq2.1}--\eqref{Eq2.5} subject to boundary conditions 
\begin{eqnarray}
\label{Eq2.7}
\varphi(0,t)=\varphi(l,t)=\psi(0,t)=\psi(l,t)=0\quad{\rm for\quad all}\quad t>0,\\
\label{Eq2.8}
y(0,t)=y(l,t)=z(0,t)=z(l,t)=0\quad{\rm for\quad all}\quad t>0,\\
\label{Eq2.9}
\theta(0,t)=\theta(l,t)=0\quad{\rm for\quad all}\quad t>0.
\end{eqnarray}
And the initial conditions are given by
\begin{eqnarray}
\label{Eq2.10}
\varphi(x,0)=\varphi_0(x),\; \varphi_t(x,0)=\varphi_1(x),\; \psi(x,0)=\psi_0(x),\quad {\rm for}\;x\in(0,l),\\
\label{Eq2.11}
\psi_t(x,0)=\psi_1(x),\; y(x,0)=y_0(x),\; y_t(x,0)=y_1(x), \quad {\rm for}\;x\in(0,l),\\
\label{Eq2.12}
z(x,0)=z_0(x),\;z_t(x,0)=z_1(x),\; \theta(x,0)=\theta_0(x),\quad{\rm for}\;x\in(0,l).
\end{eqnarray}
Note that the difference between these systems is in the coupling term in the heat equation. For the first system, the coupling is $\beta\psi_t$, which presents a derivative of zero order with respect to the spatial variable. In the second system, the coupling term is given by $\delta\psi_{txx}$, which presents a second order derivative with respect to spatial variable $x$. The coupling considered by the second system is the most common, this type of coupling is known as strong coupling. In our research it helps to show the existence of Gevrey classes and also to demonstrate the analyticity of the associated $S_2(t)$ semigroup to the second system. During the development of this investigation, we were able to observe that the zero order of the derivative in relation to the space of the coupling term for the first system was decisive for not obtaining the estimates:  $|\lambda|^\phi\|v\|^2\leq C\|F\|_{\mathbb{H}_1}\|U\|_{\mathbb{H}_1}$  and $|\lambda|^\phi\|A^\frac{1}{2}\psi\|^2\leq C\|F\|_{\mathbb{H}_1}\|U\|_{\mathbb{H}_1}$ for $0<\phi\leq 1$, which made it impossible to obtain regularity results of the first system.

 During the last decades, various investigations focused on the study of the asymptotic behavior and regularity of the Tymoshenko beam system, thermoviscoelastic Timoshenko system with diffusion effect and also of Timoshenko beam systems coupled with heat equations from Fourier's law, Cateneo's and thermoelasticity of type III. Results of exponential decay and regularity for these systems are mostly provided with dissipative terms, at least in the equations that do not refer to heat or do not have heat coupling terms. We will cite some of these works below.

In 2005,  Raposo et al.\cite{Raposo2005} studies the Timoshenko system, provided for two frictional dissipations $\varphi_t$ and $\psi_t$, and proves that the semigroup associated with the system decays exponentially. For the same Timoshenko system, when the stress-strain constitutive law is of Kelvin-Voigt type, given by
\begin{equation*}
S=\kappa(\varphi_x+\psi)+\gamma_1(\varphi_x+\psi)_t\qquad\text{and}\qquad M=b\psi_x+\psi_{xt}, 
\end{equation*} 
Malacarne A. and  Rivera J. in \cite{AMJR2016} shows that $S(t)$ is analytical if and only if the viscoelastic damping is present in both the shear stress and the bending moment. Otherwise, the corresponding semigroup is not exponentially stable no matter the choice of the coefficients. They also showed that the solution decays polynomially to zero as $t^{-1/2}$, no matter where the viscoelastic mechanism is effective and that the rate is optimal whenever the initial data are taken on the domain of the infinitesimal operator. In 2023, Su\'arez \cite{Suarez2023} studied the regularity of the model given in \cite{Raposo2005}, substituting the two damping weaks $\varphi_t$ and $\psi_t$,  for fractional dampings $(-\partial_{xx})^\tau\varphi_t$ and $(-\partial_{xx})^\sigma\psi_t$ where the parameters $(\tau,\sigma)\in [0,1]$, and proved the existence of Gevrey classes $s>\frac{r+1}{2r}$, for $r=\min\{\tau,\sigma\},\quad  \forall (\tau,\sigma)\in(0,1)^2$,  of the semigroup $S(t)$ associated to the system,  and analyticity of $S(t)$ when the two parameters $\tau$ and $\sigma$ vary in the interval $[1/2,1]$.

In 2021,  M. Elhindi and T. EL Arwadi \cite{METEA2021} studied the Timoshenko beam model with thermal, mass diffusion and viscoelastic effects:
\begin{equation}\label{Eq2.2METEA}
\left\{ \begin{array}{l}
\rho_1\varphi_{tt}-\kappa(\varphi_x-\psi)_x-\gamma_1(\varphi_x+\psi)_{xt}=0,\\
\rho_2\psi_{tt}-\alpha\psi_{xx}-\gamma_2\psi_{xxt}+\kappa(\varphi_x+\psi)+\gamma_1(\varphi_x+\psi)_t-\xi_1\theta_x-\xi_2P_x=0,\\
c\theta_t+dP_t-\kappa\theta_{xx}-\xi_1\psi_{tx}=0,
\\
d\theta_t+rP_t-hP_{xx}-\xi_2\psi_{tx}=0.
\end{array}\right.
\end{equation}
 Using semigroup theory, they proved that the considered problem is well posed with the Dirichlet boundary conditions. An exponential decay is obtained by constructing the Lyapunov functional. Finally, a numerical study based on the $P_1-$finite element approximation for spatial discretization and implicit Euler scheme for temporal discretization is carried out, where the stability of the scheme is studied, as well as error analysis and some numerical simulations are obtained. By the year of 2023, Mendes et al. in \cite{FLF2023}, present the study of the regularity of two thermoelastic beam systems defined
by the Timoshenko beam model coupled with the heat conduction of Green-Naghdiy theory of
type III; both mathematical models are differentiated by their coupling terms that arise as a
consequence of the constitutive laws initially considered. The systems presented in this work
have 3 fractional dampings: $(-\partial_{xx})^\tau \phi_t, (-\partial_{xx})^\sigma \psi_t$  and $(-\partial_{xx})^\xi \theta_t$, where $\phi,\psi$ and $\theta$ are:
transverse displacement, rotation angle and empirical temperature of the beam respectively
and the parameters $(\tau,\sigma, \xi)\in [0,1]^3$. 

The main contribution
of this article is to show that the corresponding semigroup $S_i(t) = e^{\mathcal{B}_it}$, with $i = 1,2$, is of Gevrey class $s > (r+1)/(2r)$ for $r = \min\{\tau,\sigma,\xi\}\; \forall  (\tau,\sigma,\xi) \in  (0,1)^3$. It is also showed that $S_1(t) =e^{\mathcal{B}_1t}$ is analytic in the region $RA_1 := \{(\tau,\sigma,\xi)\in [1/2, 1]^3\}$  and $S_2(t) = e^{\mathcal{B}_2t}$ is analytic in the region $RA_2 :=\{ (\tau,\sigma,\xi)\in [1/2,1]^3/ \tau=\xi\}$. Some articles published in the last decade  that study the asymptotic behavior and regularity of coupled systems and/or fractional dissipations can be consulted at \cite{AmmariShelTebou2022,DOroPata2016,KLiuH2021,FMFSSB2023, HSLiuRacke2019}.

The paper is organized as follows. In section 2, we study the well-posedness  and exponential decay of the system \eqref{Eq1.1}-\eqref{Eq1.12} through semigroup theory.  In section 3, we study the well-posedness, exponential decay, existence of Gevrey classes and analyticity of the system \eqref{Eq2.13}--\eqref{Eq2.17} with initial conditions \eqref{Eq2.10}--\eqref{Eq2.12}, for all the results we use again the semigroup theory, the good properties of fractional operator $A^r:=(-\partial_{xx})^r$ for $r\in\mathbb{R}$, a proper decomposition of the functions $u,s,w$ and the Interpolation Theorem \ref{Lions-Landau-Kolmogorov}.

%%%%%%%%%%%%%%%%%%%%%%%
\section{System 01}
In this section we present the study of the well-posedness  and the exponential decay of the first system, for both results the semigroup theory is used, the first system is given by:  
\begin{eqnarray}
\label{Eq1.1}
\rho_1 \varphi_{tt}-\kappa_1(\varphi_x-\psi)_x-\jmath(y-\varphi)+\gamma_1(-\partial_{xx})^{\tau_1}\varphi_t=0\quad{\rm in} \quad (0,l)\times (0,\infty),\\
\label{Eq1.2}
\rho_2\psi_{tt}-b_1\psi_{xx}-\kappa_1(\varphi_x-\psi)+\delta\theta_{xx}=0\quad{\rm in} \quad (0,l)\times (0,\infty),\\
\label{Eq1.3}
\rho_3y_{tt}-\kappa_2(y_x-z)_x+\jmath(y-\varphi)+\gamma_2(-\partial_{xx})^{\tau_2}y_t=0\quad{\rm in} \quad (0,l)\times (0,\infty),\\
\label{Eq1.4}
\rho_4z_{tt}-b_2z_{xx}-\kappa_2(y_x-z)+\gamma_3(-\partial_{xx})^{\tau_3}z_t=0\quad{\rm in} \quad (0,l)\times (0,\infty),\\
\label{Eq1.5}
\rho_5\theta_t-K\theta_{xx}+\beta\psi_t=0\quad{\rm in} \quad (0,l)\times (0,\infty),
\end{eqnarray}
 subject to boundary conditions 
\begin{eqnarray}
\label{Eq1.7}
\varphi(0,t)=\varphi(l,t)=\psi(0,t)=\psi(l,t)=0\quad{\rm for\quad all}\quad t>0,\\
\label{Eq1.8}
y(0,t)=y(l,t)=z(0,t)=z(l,t)=0\quad{\rm for\quad all}\quad t>0,\\
\label{Eq1.9}
\theta(0,t)=\theta(l,t)=0\quad{\rm for\quad all}\quad t>0.
\end{eqnarray}
And the initial conditions are given by
\begin{eqnarray}
\label{Eq1.10}
\varphi(x,0)=\varphi_0(x),\; \varphi_t(x,0)=\varphi_1(x),\; \psi(x,0)=\psi_0(x),\quad {\rm for}\;x\in(0,l),\\
\label{Eq1.11}
\psi_t(x,0)=\psi_1(x),\; y(x,0)=y_0(x),\; y_t(x,0)=y_1(x), \quad {\rm for}\;x\in(0,l),\\
\label{Eq1.12}
z(x,0)=z_0(x),\;z_t(x,0)=z_1(x),\; \theta(x,0)=\theta_0(x),\quad{\rm for}\;x\in(0,l).
\end{eqnarray}
%%%%%%%%%%%%%%%%%%%%%%%%
Let's define the operator $A\colon \mathfrak{D}(A) = H^2(0,l)\cap H_0^1(0,l)\to L^2(0,l)$,  such that    $A :=-\partial_{xx}$. Using this  operator $A$ the system \eqref{Eq1.1}-\eqref{Eq1.12} can be written in the following  setting
\begin{eqnarray}
\label{Eq1.13}
\rho_1 \varphi_{tt}+\kappa_1A\varphi+\kappa_1\psi_x-\jmath(y-\varphi)+\gamma_1A^{\tau_1}\varphi_t=0\quad{\rm in} \quad (0,l)\times (0,\infty),\\
\label{Eq1.14}
\rho_2\psi_{tt}+b_1A\psi-\kappa_1(\varphi_x-\psi)-\delta A\theta=0\quad{\rm in} \quad (0,l)\times (0,\infty),\\
\label{Eq1.15}
\rho_3y_{tt}+\kappa_2 Ay+\kappa_2 z_x+\jmath(y-\varphi)+\gamma_2A^{\tau_2}y_t=0\quad{\rm in} \quad (0,l)\times (0,\infty),\\
\label{Eq1.16}
\rho_4z_{tt}+b_2Az-\kappa_2(y_x-z)+\gamma_3 A^{\tau_3}z_t=0\quad{\rm in} \quad (0,l)\times (0,\infty),\\
\label{Eq1.17}
\rho_5\theta_t+KA\theta+\beta\psi_t=0\quad{\rm in} \quad (0,l)\times (0,\infty),
\end{eqnarray}
with the initial conditions \eqref{Eq1.10}--\eqref{Eq1.12}.
%%%%%%%%%%%%%%%%%%%%%%%%%%%%%%%%%%%%%%%%%%%%%%%
%%%%%%%%%   OBSERVACIÓN 01     %%%%%%%%%%%%%%%%
%%%%%%%%%%%%%%%%%%%%%%%%%%%%%%%%%%%%%%%%%%%%%%%
\begin{remark} It is known that this operator $A:=-\partial_{xx}$ is strictly positive,   selfadjoint,    has a compact inverse, and has compact resolvent.  And the operator $A^{\sigma}$ is self-adjoint positive for all $\sigma\in\R$, bounded for $\sigma\leq 0$, and  the embedding
\begin{equation*}
\mathfrak{D}(A^{\sigma_1})\hookrightarrow \mathfrak{D}(A^{\sigma_2}),
\end{equation*}
is continuous for $\sigma_1>\sigma_2$.  Here,  the norm in $\mathfrak{D}(A^{\sigma})$ is given by $\|u\|_{\mathfrak{D}(A^{\sigma})}:=\|A^{\sigma}u\|$, $u\in \mathfrak{D}(A^{\sigma})$, where $\dual{\cdot}{\cdot}$ and $\|\cdot\|$ denotes the inner product and norm in the complex Hilbert space $\mathfrak{D}(A^0)=L^2(0,l)$.  Some of the most used spaces at work are  $\mathfrak{D}(A^\frac{1}{2})=H_0^1(0,l)$ and $\mathfrak{D}(A^{-\frac{1}{2}})=H^{-1}(0,l)$.
\end{remark}
\subsection{Well-posedness of the System 01}
%%%%%%%%%%%%%%%%%%%%
Next we are going to rewrite our system \eqref{Eq1.10}--\eqref{Eq1.17} in Cauchy abstract form to apply semigroup theory: 

%%%%%%%%%%%%%%%%%%
Taking, $\varphi_t=u$, $\psi_t=v$, $y_t=s$ and $z_t=w$,   
 the initial boundary value problem \eqref{Eq1.7}-\eqref{Eq1.17} can be reduced to the following abstract initial value problem for a first-order evolution equation
 \begin{equation}\label{Fabstrata}
    \frac{d}{dt}U(t)=\mathbb{B}_i U(t),\quad    U(0)=U_0,
\end{equation}

 where $U(t)=(\varphi,u,\psi, v,y, s,z, w, \theta)^T$,  $U_0=(\varphi_0,\varphi_1,\psi_0,\psi_1, y_0,y_1,z_0,z_1,\theta_0)^T$, $i=1,2$  and   the operator $\mathbb{B}_1\colon \mathfrak{D}(\mathbb{B}_1)\subset \mathbb{H}_1\to\mathbb{H}_1$ is given by
\begin{equation}\label{operadorAgamma}
 \mathbb{B}_1U:=\left(\begin{array}{c}
 u\\
-\dfrac{\kappa_1}{\rho_1}A\varphi-\dfrac{\kappa_1}{\rho_1}\psi_x+\dfrac{\jmath}{\rho_1}(y-\varphi)-\dfrac{\gamma_1}{\rho_1}A^{\tau_1}u\\
 v\\
-\dfrac{b_1}{\rho_2}A\psi+\dfrac{\kappa_1}{\rho_2}(\varphi_x-\psi)+\dfrac{\delta}{\rho_2}A\theta\\
s\\
-\dfrac{\kappa_2}{\rho_3}Ay-\dfrac{\kappa_2}{\rho_3}z_x-\dfrac{\jmath}{\rho_3}(y-\varphi)-\dfrac{\gamma_2}{\rho_3}A^{\tau_2}s\\
w\\
-\dfrac{b_2}{\rho_4}Az+\dfrac{\kappa_2}{\rho_4}(y_x-z)-\dfrac{\gamma_3}{\rho_4}A^{\tau_3}w\\
-\dfrac{K}{\rho_5}A\theta-\dfrac{\beta}{\rho_5}v
 \end{array}\right),
\end{equation}
where,  $\mathfrak{D}(\mathbb{B}_1)$ and $\mathbb{H}_1$, will be defined next. Taking the duality product between equation \eqref{Eq1.13} and  $\varphi_t$, \eqref{Eq1.14} with $\psi_t$, \eqref{Eq1.15} with $y_t$, \eqref{Eq1.16} with $z_t$   and \eqref{Eq1.17} with $\frac{\delta}{\beta}\theta_{xx}$,  and taking advantage of the self-adjointness of the powers of the operator $A$ and as from boundary condition $z(0,t)=z(l,t)=0$, we have  $\|z_x\|^2=\dual{z_x}{z_x}=\int_0^lz_x\overline{z_x}dx=\int Az\overline{z}dx+z_x\overline{z}|_0^l=\dual{Az}{z}=\|A^\frac{1}{2}z\|^2$, similarly we have $\|\psi_x\|^2=\|A^\frac{1}{2}\psi\|^2$ and $\|\theta_x\|^2=\|A^\frac{1}{2}\theta\|^2$.   For every solution of the system \eqref{Eq1.7}-\eqref{Eq1.17}  the total energy $\mathfrak{E}_1\colon \mathbb{R}^+\to\mathbb{R}^+$ is given in the $t$ by    
\begin{multline}\label{Energia01}
\mathfrak{E}_1(t)=\frac{1}{2}\bigg[ \rho_1\|\varphi_t\|^2+\rho_2\|\psi_t\|^2+\rho_3\|y_t\|^2+\rho_4\|z_t\|^2+b_1\|A^\frac{1}{2}\psi\|^2+b_2\|A^\frac{1}{2}z\|^2
\\  +\kappa_1\|\varphi_x-\psi\|^2 +\kappa_2\|y_x-z\|^2+\jmath\|y-\varphi\|^2+\dfrac{\rho_5\delta}{\beta}\|A^\frac{1}{2}\theta\|^2 \bigg ],
\end{multline}
and satisfies
\begin{equation}\label{Dissipa01}
\dfrac{d}{dt}\mathfrak{E}_1(t)=-\gamma_1\|A^\frac{\tau_1}{2}\varphi_t\|^2-\gamma_2\|A^\frac{\tau_2}{2}y_t\|^2-\gamma_3\|A^\frac{\tau_3}{2}z_t\|^2-\dfrac{\delta K}{\beta}\|A\theta\|^2.
\end{equation}
%%%%%%%%
%}

This operator will be defined in a suitable subspace of the phase space
$$
\mathbb{H}_1:=[\mathfrak{D}(A^\frac{1}{2})\times\mathfrak{D}(A^0)]^4\times\mathfrak{D}(A^\frac{1}{2}),
$$
that is a Hilbert space with the inner product
\begin{eqnarray*}
\dual{ U_1}{U_2}_{\mathbb{H}_1} & := & \rho_1\dual{u_1}{u_2}+\rho_2\dual{v_1}{v_1}+\rho_3\dual{s_1}{s_2}+\rho_4\dual{w_1}{w_2}+b_1\dual{\psi_{1,x}}{\psi_{2,x}}\\
&  & +b_2\dual{z_{1,x}}{z_{2,x}}+\kappa_1\dual{\varphi_{1,x}-\psi_1}{\varphi_{2,x}-\psi_2} + \kappa_2\dual{y_{1,x}-z_1}{y_{2,x}-z_2}\\
& & +\jmath\dual{y_1-\varphi_1}{y_2-\varphi_2}+\dfrac{\rho_5\delta}{\beta}\dual{\theta_{1,x}}{\theta_{2,x}}.
\end{eqnarray*}
For $U_i=(\varphi_i,u_i,\psi_i, v_i,y_i, s_i,z_i,w_i,  \theta_i)^T\in \mathbb{H}_1$,  $i=1,2$  and induced norm:
\begin{multline}\label{NORM}
\|U\|_{\mathbb{H}_1}^2:=\rho_1\|u\|^2+\rho_2\|v\|^2+\rho_3\|w\|^2+\rho_4\|s\|^2+b_1\|A^\frac{1}{2}\psi\|^2+b_2\|A^\frac{1}{2}z\|^2
\\  +\kappa_1\|\varphi_x-\psi\|^2 +\kappa_2\|y_x-z\|^2+\jmath\|y-\varphi\|^2+\dfrac{\rho_5\delta}{\beta}\|A^\frac{1}{2}\theta\|^2.
\end{multline}
In these conditions, we define the domain of $\mathbb{B}_1$ as
\begin{multline}\label{dominioB}
    \mathfrak{D}(\mathbb{B}_1):= \Big \{ U\in \mathbb{H}_1 \colon  (u,v,s,w)\in  [\mathfrak{D}(A^\frac{1}{2})]^4,\theta \in \mathfrak{D}(A^\frac{1}{2})\cap H^3(0,l)\quad {\rm and}\\
     (\varphi,\psi,y,z) \in (\mathfrak{D}(A)\cap\mathfrak{D}(A^{\tau_1}))\times \mathfrak{D}(A)\times (\mathfrak{D}(A)\cap\mathfrak{D}(A^{\tau_2}))\times (\mathfrak{D}(A)\cap\mathfrak{D}(A^{\tau_3}))
       \Big\},
\end{multline}
And it is easy to verify, that
\begin{equation}\label{disipative}
{\rm Re}\dual{\mathbb{B}U}{U}_{\mathbb{H}_1}=-\gamma_1\|A^\frac{\tau_1}{2}u\|^2-\gamma_2\|A^\frac{\tau_2}{2}s\|^2-\gamma_3\|A^\frac{\tau_3}{2}w\|^2-\dfrac{\delta K}{\beta}\|A\theta\|^2\leq 0.
\end{equation}
To show that the operator $\mathbb{B}$ is the generator of a $C_0$-semigroup,  we invoke a result from Liu-Zheng \cite{LiuZ}.

\begin{theorem}[see Theorem 1.2.4 in \cite{LiuZ}] \label{TLiuZ}
Let $\mathbb{B}$ be a linear operator with domain $\mathfrak{D}(\mathbb{B})$ dense in a Hilbert space $\mathbb{H}$. If $\mathbb{B}$ is dissipative and $0\in\rho(\mathbb{B})$, the resolvent set of $\mathbb{B}$, then $\mathbb{B}$ is the generator of a $C_0$-semigroup of contractions on $\mathbb{H}$.
\end{theorem}
%%%%%%%%%
\begin{proof}
See Lemma 2.1 \cite{Ramos2023CNTs}. 
\end{proof}
%%%%%%%%%%%%%%%%%%%%%%%%%%%%%%%
 As a consequence of the previous Theorem \ref{TLiuZ},  we obtain
\begin{theorem}
Given $U_0\in\mathbb{H}$ there exists a unique weak solution $U$ to  the problem \eqref{Fabstrata} satisfying 
$$U\in C([0, +\infty), \mathbb{H}).$$
Futhermore,  if $U_0\in  \mathfrak{D}(\mathbb{B}^k), \; k\in\mathbb{N}$, then the solution $U$ of \eqref{Fabstrata} satisfies
$$U\in \bigcap_{j=0}^kC^{k-j}([0,+\infty),  \mathfrak{D}(\mathbb{B}^j).$$
\end{theorem}
%%%%%%%%%%%%%%%%%%%%%%%%%%%%%%%%%
\begin{theorem}[Hille-Yosida]\label{THY} A linear (unbounded) operator $\mathbb{B}$ is the infinitesimal generator of a $C_0-$semigroup of contractions $S(t)$, $ t\geq 0$, if and only if\\
$(i)$ $\mathbb{B}$ is closed and $\overline{\mathfrak{D}(\mathbb{B})}=\mathbb{H}$,\\
$(ii)$ the resolvent set $\rho(\mathbb{B})$ of $\mathbb{B}$ contains $\mathbb{R}^+$ and for every $\lambda>0$, 
\begin{equation*}
\|(\lambda I-\mathbb{B})^{-1}\|_{\mathcal{L}(\mathbb{H})}\leq\dfrac{1}{\lambda}.
\end{equation*}
\end{theorem}
\begin{proof}
See \cite{Pazy}.
\end{proof}
%%%%%%%%%%%%%%%%%%%%%%%%%%%%%%%%%%%%%%%%%%%%%%%%%%%%%%%%%%%%%%%%%%%%%%%%%%%%%%%%%%%%%%
\subsection{Exponential Decay of System 01,  for $(\tau_1,\tau_2,\tau_3)\in [0,1]^3$}\label{3.1}
In this section, we will study the asymptotic behavior of the semigroup of the system \eqref{Eq1.10}-\eqref{Eq1.17}.  We will use the following spectral characterization of exponential stability of semigroups due to Gearhart\cite{Gearhart} (Theorem 1.3.2  book of Liu-Zheng \cite{LiuZ}).
\begin{theorem}[see \cite{LiuZ}]\label{LiuZExponential}
Let $S(t)=e^{t\mathbb{B}}$ be  a  $C_0$-semigroup of contractions on  a Hilbert space $ \mathbb{H}$. Then $S(t)$ is exponentially stable if and only if  
	\begin{equation}\label{EImaginario}
\rho(\mathbb{B})\supseteq\{ i\lambda;  \lambda\in \R \} 	\equiv i\R
\end{equation}
and
\begin{equation}\label{Exponential}
 \limsup\limits_{|\lambda|\to
   \infty}   \|(i\lambda I-\mathbb{B})^{-1}\|_{\mathcal{L}( \mathbb{H})}<\infty
\end{equation}
holds.
\end{theorem}
%%%%%%%  Remark 4
\begin{remark}\label{EqvExponential}
 Note that to show the condition \eqref{Exponential} for system 01: \eqref{Eq1.10}-\eqref{Eq1.17}, it is enough to show that: Let $\delta>0$. There exists a constant $C_\delta>0$ such that the solutions of the system \eqref{Eq1.10}-\eqref{Eq1.17} for $|\lambda|>\delta$,  satisfy the inequality 
 \begin{equation}\label{EqvExponencial}
 \|U\|_{\mathbb{H}_1}\leq C_\delta\|F\|_{\mathbb{H}_1}\qquad {\rm for}\quad 0\leq\tau_1,\tau_2,\tau_3\leq 1.
 \end{equation}
 \end{remark}
%%%%%%%%%%%%%%%%%%%
%%%%%%%%%%%%%%%%%%%
To use Theorem \ref{LiuZExponential},  we will try to obtain some estimates  for
 $$U=(\varphi,u,\psi, v,y, s,z,w, \theta)^T\in \mathfrak{D}(\mathbb{B}_1)\;{\rm and}\; F=(f^1, f^2, f^3, f^4, f^5, f^6, f^7, f^8, f^9)^T\in \mathbb{H}_1,$$  such that $(i\lambda I-\mathbb{B}_1)U=F$, where $\lambda\in \R$. This system, written in components,  reads
%%% % % % % % % % % %% % % % % % % % % % % % % % % % % %
%%%%%%%%%%%%%%%%%%%%%%%%%%%%%%%%%%%%%%%%%%%%%%%%%%
\begin{eqnarray}
i\lambda \varphi-u &=& f^1\quad {\rm in}\quad \mathfrak{D}(A^\frac{1}{2})\label{Pesp-10}\\
i\lambda u+\dfrac{\kappa_1}{\rho_1}A\varphi+\dfrac{\kappa_1}{\rho_1}\psi_x-\dfrac{\jmath}{\rho_1}(y-\varphi)+\dfrac{\gamma_1}{\rho_1}A^{\tau_1}u &=& f^2\quad {\rm in}\quad \mathfrak{D}(A^0)\label{Pesp-20}\\
i\lambda\psi-v &= & f^3\quad {\rm in}\quad \mathfrak{D}(A^\frac{1}{2})\label{Pesp-30}\\
i\lambda v+\dfrac{b_1}{\rho_2}A\psi-\dfrac{\kappa_1}{\rho_2}(\varphi_x-\psi)-\dfrac{\delta}{\rho_2}A\theta &=& f^4\quad {\rm in}\quad \mathfrak{D}(A^0) \label{Pesp-40}\\
i\lambda y-s &= &  f^5\quad {\rm in}\quad \mathfrak{D}(A^\frac{1}{2})\label{Pesp-50}\\
i\lambda s+\dfrac{\kappa_2}{\rho_3}Ay+\dfrac{\kappa_2}{\rho_3}z_x+\dfrac{\jmath}{\rho_3}(y-\varphi)+\dfrac{\gamma_2}{\rho_3}A^{\tau_2} s &=& f^6\quad {\rm in}\quad \mathfrak{D}(A^0) \label{Pesp-60}\\
i\lambda z-w &= & f^7\quad {\rm in}\quad \mathfrak{D}(A^\frac{1}{2})\label{Pesp-70}\\
i\lambda w+\dfrac{b_2}{\rho_4}Az-\dfrac{\kappa_2}{\rho_4}(y_x-z)+\dfrac{\gamma_3}{\rho_4}A^{\tau_3}w &= & f^8\quad {\rm in}\quad \mathfrak{D}(A^0) \label{Pesp-80}\\
i\lambda\theta+\dfrac{K}{\rho_5} A\theta+\dfrac{\beta}{\rho_5} v &=& f^9 \quad {\rm in}\quad \mathfrak{D}(A^\frac{1}{2}). \label{Pesp-90}
\end{eqnarray}
From \eqref{disipative},   we have the first estimate
\begin{multline*}
|\gamma_1\|A^\frac{\tau_1}{2}u\|^2+\gamma_2\|A^\frac{\tau_2}{2}s\|^2+\gamma_3\|A^\frac{\tau_3}{2}w\|^2+\dfrac{\delta K}{\beta}\|A\theta\|^2|\\
 = |-{\rm Re}\dual{\mathbb{B}U}{U}|=|{\rm Re}\{ \dual{i\lambda U-F}{U}  \}|\\
\leq  |\dual{F}{U}|\leq \|F\|_{\mathbb{H}_1}\|\|U\|_{\mathbb{H}_1}.
\end{multline*}
Therefore
\begin{equation}\label{dis-10}
\gamma_1\|A^\frac{\tau_1}{2}u\|^2+\gamma_2\|A^\frac{\tau_2}{2}s\|^2+\gamma_3\|A^\frac{\tau_3}{2}w\|^2+\dfrac{\delta K}{\beta}\|A\theta\|^2 \leq  \|F\|_{\mathbb{H}_1}\|\|U\|_{\mathbb{H}_1}.
\end{equation}
%%%%%%%%%%%%%%%%%%%%%%%%%%%%%%%%%5
Next, we show some lemmas that will lead us to the proof of the main theorem of this section.
%%%%%%%%%%%%%%%%%%%%%%%%%%%%%
%%%%   Lemma  08    01 %%%%%%
%%%%%%%%%%%%%%%%%%%%%%%%%%%%%
\begin{lemma}\label{Lemma3}
Let $\delta>0$.  There exists $C_\delta>0$ such that the solutions of the system \eqref{Eq1.10}-\eqref{Eq1.17}  for  $(\tau_1,\tau_2,\tau_3)\in[0,1]^3$, and for $\varepsilon>0$, exists $C_\varepsilon>0$ independent of $\lambda$, such that
\begin{eqnarray}\label{Item01Lemma3}
\|v\|^2 & \leq &C_\varepsilon \|F\|_{\mathbb{H}_1}\|U\|_{\mathbb{H}_1}+\varepsilon\|U\|^2_{\mathbb{H}_1}.
\end{eqnarray}
\end{lemma}
%%%%%%%%  proof  %%%%
\begin{proof} Applying  the duality product between \eqref{Pesp-90} and $v$ and using \eqref{Pesp-40}, we have
\begin{eqnarray*}
\dfrac{\beta}{\rho_5}\|v\|^2=\dual{\theta}{i\lambda v}-\dfrac{K}{\rho_5}\dual{A\theta}{v}+\dual{f^9}{v}\\
=\dual{\theta}{-\dfrac{b_1}{\rho_2}A\psi+\dfrac{\kappa_1}{\rho_2}(\varphi_x-\psi)+\dfrac{\delta}{\rho_2}A\theta+f^4}-\dfrac{K}{\rho_5}\dual{A\theta}{v}+\dual{f^9}{v}\\
=-\dfrac{b_1}{\rho_2}\dual{A^\frac{1}{2}\theta}{A^\frac{1}{2}\psi}+\dfrac{\kappa_1}{\rho_2}\dual{\theta}{(\varphi_x-\psi)}+\dfrac{\delta}{\rho_2}\|A^\frac{1}{2}\theta\|^2+\dual{\theta}{f^4}-\dfrac{K}{\rho_5}\dual{A\theta}{v}+\dual{f^9}{v}.
\end{eqnarray*}
Applying Cauchy-Schwarz  and Young inequalities, continuous immersions: $\mathfrak{D}(A)\hookrightarrow \mathfrak{D}(A^\frac{1}{2})\hookrightarrow\mathfrak{D}(A^0)$, for $\varepsilon>0$, exists $C_\varepsilon>0$, such that
\begin{equation*}
\|v\|^2\leq C_\varepsilon\|A\theta\|^2+\varepsilon\{ \|A^\frac{1}{2}\psi\|^2+ \|\varphi_x-\psi\|^2+\|v\|^2\}+\|\theta\|\|f^4\|+\|f^9\|\|v\|,
\end{equation*}
 from  estimative \eqref{dis-10},   finish proof this item.
\end{proof}
%%%%%%%%%%%%%%%%%%%%%%%%%%%%%
%%%%%%%%%%%%%%%%%%%%%%%%%%%%%
%%%%   Lemma  09   01  %%%%%%
%%%%%%%%%%%%%%%%%%%%%%%%%%%%%
\begin{lemma}\label{Lemma2}
Let $\delta>0$.  There exists $C_\delta>0$ such that the solutions of the system \eqref{Eq1.10}-\eqref{Eq1.17}  for $|\lambda| > \delta$ and $(\tau_1,\tau_2,\tau_3)\in [0,1]^3$,  satisfy
\begin{eqnarray}
\label{Item01Lemma2}
(i)\quad|\lambda|\| y-\varphi\|^2  \leq  C_\delta\|F\|_{\mathbb{H}_1}\|U\|_{\mathbb{H}_1},\\
\label{Item02Lemma2}
(ii)\quad \kappa_1\|\varphi_x-\psi\|^2+b_1\|A^\frac{1}{2}\psi\|^2\leq \varepsilon\|U\|^2_{\mathbb{H}_1}+C_\delta\|F\|_{\mathbb{H}_1}\|U\|_{\mathbb{H}_1},\\
\label{Item03Lemma2}
(iii)\quad \kappa_2\|y_x-z\|^2+b_2\|A^\frac{1}{2}z\|^2\leq C_\delta\|F\|_{\mathbb{H}_1}\|U\|_{\mathbb{H}_1}.
\end{eqnarray}
\end{lemma}
\begin{proof}$\! (i)$ Making the difference between the equations \eqref{Pesp-50} and \eqref{Pesp-10}, we have 
$$i\lambda(y-\varphi)-(s-u)=f^5-f^1.$$
 taking the duality product between this last equation and $y-\varphi$, we arrive at:
\begin{equation}\label{Eq01Lemma3}
i\lambda\|y-\varphi\|^2=\dual{s}{y-\varphi}-\dual{u}{y-\varphi}+\dual{f^5}{y-\varphi}-\dual{f^1}{y-\varphi}.
\end{equation}
Applying Cauchy-Schwarz and Young inequalities, norms $\|F\|_{\mathbb{H}_1}$ and $\|U\|_{\mathbb{H}_1}$, and for $\varepsilon>0$, exists $C_\varepsilon>0$, such that
\begin{equation}\label{Eq02Lemma3}
|\lambda| \|y-\varphi\|^2\leq C_\varepsilon\{ \|s\|^2+\|u\|^2\}+\varepsilon\|y-\varphi\|^2+C_\delta\|F\|_{\mathbb{H}_1}\|U\|_{\mathbb{H}_1}.
\end{equation}
 finally the estimative \eqref{dis-10} and considering $|\lambda|>\delta>1$,  we finish proof of item $(i)$.\\
%%%%%%%%%%%%%%%%%%%%%%%%%%%%%%%%%%%%%%%%
$(ii)$ Performing the duality product of \eqref{Pesp-20} for $\rho_1\varphi$ and using \eqref{Pesp-10}, we obtain
\begin{multline*}
\kappa_1\dual{(\varphi_x-\psi)}{\varphi_x}=\rho_1\dual{u}{i\lambda\varphi}+\jmath\dual{(y-\varphi)}{u}-\gamma_1\dual{A^{\tau_1}u}{\varphi}+\rho_1\dual{f^2}{\varphi}\\
=\rho_1\|u\|^2+\rho_1\dual{u}{f^1}+\jmath\dual{(y-\varphi)}{u}-i\lambda\gamma_1\|A^\frac{\tau_1}{2}\varphi\|^2\\
+\gamma_1\dual{A^\frac{\tau_1}{2}f^1}{A^\frac{\tau_1}{2}\varphi}
+\rho_1\dual{f^2}{\varphi},
\end{multline*}
now, performing the duality product of \eqref{Pesp-40} for $\rho_2\psi$ and using \eqref{Pesp-30}, we obtain
\begin{equation*}
\kappa_1\dual{(\varphi_x-\psi)}{\psi}=-\rho_2\|v\|^2-\rho_2\dual{v}{f^3}+b_1\|A^\frac{1}{2}\psi\|^2-\delta\dual{A\theta}{\psi}-\rho_2\dual{f^4}{\psi},
\end{equation*}
subtracting the last two equations, we have
\begin{multline}
\label{Eq02Lemma2}
\kappa_1\|\varphi_x-\psi\|^2+b_1\|A^\frac{1}{2}\psi\|^2=\rho_2\|v\|^2+\rho_1\|u\|^2+\rho_1\dual{u}{f^1}+\jmath\dual{(y-\varphi)}{u}-i\lambda\gamma_1\|A^\frac{\tau_1}{2}\varphi\|^2\\
+\gamma_1\dual{A^\frac{\tau_1}{2}f^1}{A^\frac{\tau_1}{2}\varphi}
+\rho_1\dual{f^2}{\varphi}+\rho_2\dual{v}{f^3}+\delta\dual{A\theta}{\psi}+\rho_2\dual{f^4}{\psi}.
\end{multline}
Taking real part in \eqref{Eq02Lemma2},   applying Cauchy-Schwarz and Young inequalities, estimates \eqref{dis-10}, Lemma \ref{Lemma3} and \eqref{Item01Lemma2} (item $(i)$  in this lemma), we finish proof of item $(ii)$.\\
%%%%%%%%%%%%%%%%%%%%%%%%%%%%%%%%%%%%%%%%%%%%%%%%%%%
 $(iii)$ Performing the duality product of \eqref{Pesp-60} for $\rho_3y$ and using \eqref{Pesp-50}, we obtain
\begin{multline*}
\kappa_2\dual{(y_x-z)}{y_x}=\rho_3\dual{s}{i\lambda y}-\jmath\dual{(y-\varphi)}{y}-\gamma_2\dual{A^{\tau_2}s}{y}+\rho_3\dual{f^6}{y}\\
=\rho_3\|s\|^2+\rho_3\dual{s}{f^5}-\jmath\dual{(y-\varphi)}{y}-i\lambda\gamma_2\|A^\frac{\tau_2}{2}y\|^2\\
+\gamma_2\dual{A^\frac{\tau_2}{2}f^1}{A^\frac{\tau_2}{2}y}
+\rho_3\dual{f^6}{y}\\
=\rho_3\|s\|^2+\rho_3\dual{s}{f^5}-\dfrac{i\jmath}{\lambda}\dual{(y-\varphi)}{f^5}-\dfrac{i\jmath}{\lambda}\dual{(y-\varphi)}{s}\\
-i\lambda\gamma_2\|A^\frac{\tau_2}{2}y\|^2+\gamma_2\dual{A^\frac{\tau_2}{2}f^1}{A^\frac{\tau_2}{2}y}
+\rho_3\dual{f^6}{y}
\end{multline*}
now, performing the duality product of \eqref{Pesp-80} for $\rho_4z$ and using \eqref{Pesp-70}, we obtain
\begin{multline*}
\kappa_2\dual{(y_x-z)}{z}=-\rho_4\|w\|^2-\rho_4\dual{w}{f^7}+b_2\|A^\frac{1}{2}z\|^2+i\lambda\gamma_3\|A^\frac{\tau_3}{2}z\|^2\\
-\gamma_3\dual{A^\frac{\tau_3}{2} f^7}{A^\frac{\tau_3}{2}z}-\rho_4\dual{f^8}{z},
\end{multline*}
subtracting the last two equations, we have
\begin{multline}
\label{Eq03Lemma2}
\hspace*{-0.25cm}\kappa_2\|y_x-z\|^2+b_2\|A^\frac{1}{2}z\|^2=\rho_3\|s\|^2+\rho_4\|w\|^2 +\rho_3\dual{s}{f^5}-\dfrac{i\jmath}{\lambda}\dual{(y-\varphi)}{f^5}-\dfrac{i\jmath}{\lambda}\dual{(y-\varphi)}{s}\\
-i\lambda\gamma_2\|A^\frac{\tau_2}{2}y\|^2+\gamma_2\dual{A^\frac{\tau_2}{2}f^1}{A^\frac{\tau_2}{2}y}
+\rho_3\dual{f^6}{y}+\rho_4\dual{w}{f^7}\\
-i\lambda\gamma_3\|A^\frac{\tau_3}{2}z\|^2+\gamma_3\dual{A^\frac{\tau_3}{2}f^7}{A^\frac{\tau_3}{2}z}+\rho_4\dual{f^8}{z}.
\end{multline}
Taking real part in \eqref{Eq03Lemma2},   applying Cauchy-Schwarz and Young inequalities, estimative \eqref{dis-10},  norms $\|F\|_{\mathbb{H}_1}$ and $\|U\|_{\mathbb{H}_1}$  and item $(i)$ this lemma,  we finish proof of item $(iii)$.
\end{proof}
%%%%%%%%%%%%%%%%%%%%%%%%%%%
%%%%%%    Theorem  Decaimento Exponencial    %%%%%%%%%%%%%
\begin{theorem}\label{TDExponential}
The semigroup $S_1(t) = e^{t\mathbb{B}_1}$, is exponentially stable as long as the parameters $(\tau_1,\tau_2,\tau_3)\in [0,1]^ 3 $.
\end{theorem}
\begin{proof} 
Let's first check the condition \eqref{EqvExponencial}, which implies \eqref{Exponential}. Using  the Lemmas \ref{Lemma3},   \ref{Lemma2}  and applying in the sequence the estimates of \eqref{dis-10}, we arrive at:
\begin{equation}\label{Eq012Exponential}
\|U\|_\mathbb{H}^2\leq C_\delta \|F\|_{\mathbb{H}_1}\|U\|_{\mathbb{H}_1}\quad{\rm for}\quad 0\leq\tau_1,\tau_2,\tau_3\leq 1.
\end{equation}
Therefore the condition \eqref{Exponential} for $(\tau_1,\tau_2,\tau_3)\in [0,1]^3$ of Theorem \ref{LiuZExponential} is verified.  Next, we show the condition \eqref{EImaginario}.
%%%%%%%%%%%%%%%%%%%%%%%%%%%%%%%%%%
%%  Lemma 10:  Eixo Imaginário  %%
%%%%%%%%%%%%%%%%%%%%%%%%%%%%%%%%%%
\begin{lemma}\label{EImaginary}
\label{iR}
Let $\varrho(\mathbb{B}_1)$ be the resolvent set of operator
$\mathbb{B}_1$. Then
\begin{equation}
i\hspace{0.5pt}\mathbb{R}\subset\varrho(\mathbb{B}_1).
\end{equation}
\end{lemma}
%%%%%%%
\begin{proof}
Since $\mathbb{B}_1$   is the infinitesimal generator of a $C_0-$semigroup of contractions $S_1(t)$, $ t\geq 0$,  from Theorem \ref{THY},  $\mathbb{B}_1$ is a closed operator and $\mathfrak{D}(\mathbb{B}_1)$ has compact embedding into the energy space $\mathbb{H}_1$, the spectrum $\sigma(\mathbb{B}_1)$ contains only eigenvalues.
 Let us prove that $i\R\subset\rho(\mathbb{B}_1)$ by using an argument by contradiction, so we suppose that $i\R\not\subset \rho(\mathbb{B}_1)$. 
 As $0\in\rho(\mathbb{B}_1)$ and $\rho(\mathbb{B}_1)$ is open, we consider the highest positive number $\lambda_0$ such that the $(-i\lambda_0,i\lambda_0)\subset\rho(\mathbb{B}_1)$ then $i\lambda_0$ or $-i\lambda_0$ is an element of the spectrum $\sigma(\mathbb{B}_1)$.   
 We suppose $i\lambda_0\in \sigma(\mathbb{B}_1)$ (if $-i\lambda_0\in \sigma(\mathbb{B}_1)$ the proceeding is similar). Then, for $0<\delta<\lambda_0$ there exist a sequence of real numbers $(\lambda_n)$, with $\delta\leq\lambda_n<\lambda_0$, $\lambda_n\con \lambda_0$, and a vector sequence  $U_n=(u_n,v_n,w_n, \theta_n)\in \mathfrak{D}(\mathbb{B}_1)$ with  unitary norms, such that
\begin{equation*}
\|(i\lambda_nI-\mathbb{B}_1) U_n\|_{\mathbb{H}_1}=\|F_n\|_{\mathbb{H}_1}\con 0,
\end{equation*}
as $n\con \infty$.   From  estimative \eqref{Eq012Exponential},   we have 
\begin{multline}
\|U_n\|^2_{\mathbb{H}_1} =\rho_1\|u\|^2+\rho_2\|v\|^2+\rho_3\|w\|^2+\rho_4\|s\|^2+b_1\|A^\frac{1}{2}\psi\|^2+b_2\|A^\frac{1}{2}z\|^2
\\  +\kappa_1\|\varphi_x-\psi\|^2 +\kappa_2\|y_x-z\|^2+\jmath\|y-\varphi\|^2+\dfrac{\rho_5\delta}{\beta}\|A^\frac{1}{2}\theta\|^2\\
\leq C_\delta\|F_n\|_{\mathbb{H}_1}\|U_n\|_{\mathbb{H}_1}=C_\delta\|F_n\|_{\mathbb{H}_1}\con 0.
\end{multline}
Therefore, we have  $\|U_n\|_{\mathbb{H}_1}\con 0$ but this is an absurd, since $\|U_n\|_{\mathbb{H}_1}=1$ for all $n\in\N$. Thus, $i\R\subset \rho(\mathbb{B}_1)$. This completes the proof of this lemma. 
\end{proof}
%%%%%%%%%%%%%%%%%
Therefore the semigroup $S_1(t)=e^{t\mathbb{B}_1}$ is exponentially stable  for $(\tau_1,\tau_2,\tau_3)\in [0,1]^3$,  thus we finish the proof of this Theorem \ref{TDExponential}. 
\end{proof}
%%%%%%%%%%%%%%%%%%%%%%%%%%%%  Final del  decaimiento Exponencial   %%%%%%%%%%%%%%%%%%%

%%%%%%%%%%%%%%%%%%%%%%%%%%%%%%%%%%%%%%%%%%%%%%%%%%%%%%%%%%%%%%%%%%%%%%%%%%%%%%%
\section{System 02}
In this section we present results of asymptotic behavior (Exponential Decay) and regularity (Determination of Gevrey Classes and Analyticity) of the second system of this research.
\subsection{Well-posedness of the System 02}
%%%%%%%%%%%%%%%%%%%%%%%%
Now, using  the operator $A:=-\partial_{xx}$ the system \eqref{Eq2.1}-\eqref{Eq2.12} can be written in the following  setting
\begin{eqnarray}
\label{Eq2.13}
\rho_1 \varphi_{tt}+\kappa_1A\varphi+\kappa_1\psi_x-\jmath(y-\varphi)+\gamma_1A^{\beta_1}\varphi_t=0\quad{\rm in} \quad (0,l)\times (0,\infty),\\
\label{Eq2.14}
\rho_2\psi_{tt}+b_1A\psi-\kappa_1(\varphi_x-\psi)-\delta A\theta=0\quad{\rm in} \quad (0,l)\times (0,\infty),\\
\label{Eq2.15}
\rho_3y_{tt}+\kappa_2 Ay+\kappa_2 z_x+\jmath(y-\varphi)+\gamma_2A^{\beta_2}y_t=0\quad{\rm in} \quad (0,l)\times (0,\infty),\\
\label{Eq2.16}
\rho_4z_{tt}+b_2Az-\kappa_2(y_x-z)+\gamma_3 A^{\beta_3}z_t=0\quad{\rm in} \quad (0,l)\times (0,\infty),\\
\label{Eq2.17}
\rho_5\theta_t+KA\theta+\delta A\psi_t=0\quad{\rm in} \quad (0,l)\times (0,\infty).
\end{eqnarray}
with the initial conditions; \eqref{Eq2.10}--\eqref{Eq2.12}.

%%%%%%%%%%%%%%%%%%%%%%%%%%%%%%%%%%%%%%%%%%%%%%%
%%%%%%%%%   OBSERVACIÓN 01     %%%%%%%%%%%%%%%%
%%%%%%%%%%%%%%%%%%%%%%%%%%%%%%%%%%%%%%%%%%%%%%%
Taking the duality product between equation \eqref{Eq2.13} and  $\varphi_t$, \eqref{Eq2.14} with $\psi_t$, \eqref{Eq2.15} with $y_t$, \eqref{Eq2.16} with $z_t$   and \eqref{Eq2.17} with $\theta$,   taking advantage of the self-adjointness of the powers of the operator $A$,  with boundary condition $z(0,t)=z(l,t)=0$, we have  $\|z_x\|^2=\dual{z_x}{z_x}=\int_0^lz_x\overline{z_x}dx=\int Az\overline{z}dx+z_x\overline{z}|_0^l=\dual{Az}{z}=\|A^\frac{1}{2}z\|^2$, similarly we have $\|\psi_x\|^2=\|A^\frac{1}{2}\psi\|^2$ and $\|\theta_x\|^2=\|A^\frac{1}{2}\theta\|^2$.   For every solution of the system \eqref{Eq1.7}-\eqref{Eq1.17}  the total energy $\mathfrak{E}_2\colon \mathbb{R}^+\to\mathbb{R}^+$ is given by    
\begin{multline}\label{Energia02}
\mathfrak{E}_2(t)=\frac{1}{2}\bigg[ \rho_1\|\varphi_t\|^2+\rho_2\|\psi_t\|^2+\rho_3\|y_t\|^2+\rho_4\|z_t\|^2+b_1\|A^\frac{1}{2}\psi\|^2+b_2\|A^\frac{1}{2}z\|^2
\\  +\kappa_1\|\varphi_x-\psi\|^2 +\kappa_2\|y_x-z\|^2+\jmath\|y-\varphi\|^2+\rho_5\|\theta\|^2 \bigg ],
\end{multline}
and satisfies
\begin{equation}\label{Dissipa02}
\dfrac{d}{dt}\mathfrak{E}_2(t)=-\gamma_1\|A^\frac{\beta_1}{2}\varphi_t\|^2-\gamma_2\|A^\frac{\beta_2}{2}y_t\|^2-\gamma_3\|A^\frac{\beta_3}{2}z_t\|^2- K\|A^\frac{1}{2}\theta\|^2.
\end{equation}
%%%%%%%%
%}
%%%%%%%%%%%%%%%%%%
Taking $\varphi_t=u$, $\psi_t=v$, $y_t=s$ and $z_t=w$,   
 the initial boundary value problem \eqref{Eq2.13}-\eqref{Eq2.17} can be reduced to the following abstract initial value problem for a first-order evolution equation
 \begin{equation}\label{Fabstrata2}
    \frac{d}{dt}U(t)=\mathbb{B}_2 U(t),\quad    U(0)=U_0,
\end{equation}

 where $U(t)=(\varphi,u,\psi, v,y, s,z, w, \theta)^T$,  $U_0=(\varphi_0,\varphi_1,\psi_0,\psi_1, y_0,y_1,z_0,z_1,\theta_0)^T$  and   the operator $\mathbb{B}_2\colon \mathfrak{D}(\mathbb{B}_2)\subset \mathbb{H}_2\to\mathbb{H}_2$ is given by
\begin{equation}\label{operadorAgamma2}
 \mathbb{B}_2U:=\left(\begin{array}{c}
 u\\
 -\dfrac{\kappa_1}{\rho_1}A\varphi-\dfrac{\kappa_1}{\rho_1}\psi_x+\dfrac{\jmath}{\rho_1}(y-\varphi)-\dfrac{\gamma_1}{\rho_1}A^{\beta_1}u\\
 v\\
-\dfrac{b_1}{\rho_2}A\psi+\dfrac{\kappa_1}{\rho_2}(\varphi_x-\psi)+\dfrac{\delta}{\rho_2}A\theta\\
s\\
-\dfrac{\kappa_2}{\rho_3}Ay-\dfrac{\kappa_2}{\rho_3}z_x-\dfrac{\jmath}{\rho_3}(y-\varphi)-\dfrac{\gamma_2}{\rho_3}A^{\beta_2}s\\
w\\
-\dfrac{b_2}{\rho_4}Az+\dfrac{\kappa_2}{\rho_4}(y_x-z)-\dfrac{\gamma_3}{\rho_4}A^{\beta_3}w\\
-\dfrac{K}{\rho_5}A\theta-\dfrac{\delta}{\rho_5}Av
 \end{array}\right).
\end{equation}
This operator will be defined in a suitable subspace of the phase space
$$
\mathbb{H}_2:=[\mathfrak{D}(A^\frac{1}{2})\times\mathfrak{D}(A^0)]^4\times\mathfrak{D}(A^0).
$$
It is a Hilbert space with the inner product
\begin{eqnarray*}
\dual{ U_1}{U_2}_{\mathbb{H}_2} & := & \rho_1\dual{u_1}{u_2}+\rho_2\dual{v_1}{v_1}+\rho_3\dual{s_1}{s_2}+\rho_4\dual{w_1}{w_2}+b_1\dual{\psi_{1,x}}{\psi_{2,x}}\\
&  & +b_2\dual{z_{1,x}}{z_{2,x}}+\kappa_1\dual{\varphi_{1,x}-\psi_1}{\varphi_{2,x}-\psi_2} + \kappa_2\dual{y_{1,x}-z_1}{y_{2,x}-z_2}\\
& & +\jmath\dual{y_1-\varphi_1}{y_2-\varphi_2}+\rho_5\delta\dual{\theta_1}{\theta_2}.
\end{eqnarray*}
For $U_i=(\varphi_i,u_i,\psi_i, v_i,y_i, s_i,z_i,w_i,  \theta_i)^T\in \mathbb{H}_2$,  $i=1,2$  and induced norm
\begin{multline}\label{NORM2}
\|U\|^2_{\mathbb{H}_2}:=\rho_1\|u\|^2+\rho_2\|v\|^2+\rho_3\|w\|^2+\rho_4\|s\|^2+b_1\|A^\frac{1}{2}\psi\|^2+b_2\|A^\frac{1}{2}z\|^2
\\  +\kappa_1\|\varphi_x-\psi\|^2 +\kappa_2\|y_x-z\|^2+\jmath\|y-\varphi\|^2+\rho_5\|\theta\|^2.
\end{multline}
In these conditions, we define the domain of $\mathbb{B}_2$ as
\begin{multline}\label{dominioB2}
    \mathfrak{D}(\mathbb{B}_2):= \Big \{ U\in \mathbb{H}_2 \colon  (u,v,s,w)\in  [\mathfrak{D}(A^\frac{1}{2})]^4,\theta \in \mathfrak{D}(A)\quad {\rm and}\\
     (\varphi,\psi,y,z) \in (\mathfrak{D}(A)\cap\mathfrak{D}(A^{\beta_1}))\times \mathfrak{D}(A)\times (\mathfrak{D}(A)\cap\mathfrak{D}(A^{\beta_2}))\times (\mathfrak{D}(A)\cap\mathfrak{D}(A^{\beta_3}))
       \Big\}.
\end{multline}

To show that the operator $\mathbb{B}_2$ is the generator of a $C_0$-semigroup we invoke a result from Liu-Zheng' \cite{LiuZ}. Theorem \ref{TLiuZ}:
Clearly, we see that $\mathfrak{D}(\mathbb{B}_2)$ is dense in $\mathbb{H}_2$. And  it is easy to see that  $\mathbb{B}_2$ is dissipative. In fact, for each $U=(\varphi,u,\psi, v,y, s,z, w, \theta)^T\in \mathfrak{D}(\mathbb{B}_2)$ we have
\begin{equation}\label{disipative2}
{\rm Re}\dual{\mathbb{B}_2U}{U}_{\mathbb{H}_2}=-\gamma_1\|A^\frac{\beta_1}{2}u\|^2-\gamma_2\|A^\frac{\beta_2}{2}s\|^2-\gamma_3\|A^\frac{\beta_3}{2}w\|^2- K\|A^\frac{1}{2}\theta\|^2\leq 0.
\end{equation}
Therefore, it is enough to show that  $0\in \rho(\mathbb{B}_2)$ (resolvent set of $\mathbb{B}_2$), hence we must show that $(0I-\mathbb{B}_2)^{-1}$ exists and is bounded in $\mathbb{H}_2$. 

To do that, let us take $F=(f^1, f^2, f^3,f^4,f^5,f^6,f^7, f^8, f^9)^T\in \mathbb{H}_2$, and look a unique $U=(\varphi,u,\psi, v,y, s,z, w, \theta)^T\in \mathfrak{D}(\mathbb{B}_2)$,  such that
\begin{equation}\label{39Ramos}
-\mathbb{B}_2U=F, \qquad {\rm in}\qquad \mathbb{H}_2.
\end{equation}
Equivalently, we get $-u=f^1, -v=f^3, -s=f^5, -w=f^7$, $A\theta=\frac{\delta}{K}Af^3+\frac{\rho_5}{K}f^9$ and the followings equations
\begin{eqnarray}
\label{40Ramos}
\kappa_1 A\varphi+\kappa_1\psi_x-\jmath (y-\varphi) & = &\gamma_1 A^{\beta_1}f^1+\rho_1 f^2,\quad \text{in}\quad D(A^\frac{1}{2})\\
\label{41Ramos}
b_1A\psi-\kappa_1(\varphi_x-\psi) & = & \dfrac{\delta^2}{K}Af^3+\rho_2 f^4+ \dfrac{\delta\rho_5}{K} f^9,\quad \text{in}\quad D(A^\frac{1}{2})\\
\label{42Ramos}
\kappa_2 Ay+\kappa_2 z_x+\jmath (y-\varphi) & = & \gamma_2 A^{\beta_2}f^5+\rho_3f^6,\quad \text{in}\quad D(A^\frac{1}{2})
\\
\label{43Ramos}
b_2 Az-\kappa_2 (y_x-z) & = & \gamma_3 A^{\beta_3}f^7+\rho_4 f^8, \quad \text{in}\quad D(A^\frac{1}{2}).
\end{eqnarray}

Perform the duality product of  \eqref{40Ramos}--\eqref{43Ramos} with
 $\varphi^*,\psi^*,y^*$ and $z^*$ respectively, and adding,  and using identities   $\dual{A\varphi}{\varphi^*}=\dual{\varphi_x}{\varphi_x^*}$, $\dual{A\psi}{\psi^*}=\dual{\psi_x}{\psi_x^*}$,$\dual{Az}{z^*}=\dual{z_x}{z_x^*}$ and $\dual{Ay}{y^*}=\dual{y_x}{y_x^*}$,   we obtain the equivalent variational problem:
 \begin{equation}\label{44Ramos}
 \mathfrak{B}((\varphi,\psi,y,z),(\varphi^*,\psi^*,y^*, z^*)) =\mathfrak{L}(\varphi^*,\psi^*, y^*,z^*),
 \end{equation}
 where $\mathfrak{B}(\cdot,\cdot)$ is the sesquilinear form in $[D(A^\frac{1}{2})]^4$, given by
\begin{eqnarray}
\nonumber
\mathfrak{B}((\varphi,\psi,y,z),(\varphi^*,\psi^*,y^*, z^*)) & = & \kappa_1\dual{\varphi_x}{\varphi_x^*}-\kappa_1\dual{\psi}{\varphi_x^*}+b_1\dual{\psi_x}{\psi_x^*}-\kappa_1\dual{\varphi_x-\psi}{\psi^*} 
\\
\nonumber
& & +\kappa_2\dual{y_x}{y_x^*}-\kappa_2\dual{z}{y_x^*}-\kappa_2\dual{y_x-z}{z^*}\\
\nonumber
& & 
-\jmath\dual{y-\varphi}{\varphi^*}+\jmath\dual{y-\varphi}{y^*}+b_2\dual{z_ x}{z_x^*}\\
\nonumber
&=&\kappa_1\dual{\varphi_x-\psi}{\varphi_x^*-\psi^*}+\kappa_2\dual{y_x-z}{y_x^*-z^*}+b_1\dual{\psi_x}{\psi_x^*}\\
\label{45Ramos}
& & +b_2\dual{z_x}{z_x^*}+\jmath\dual{y-\varphi}{y^*-\varphi^*}
\end{eqnarray}
and $\mathfrak{L}(\cdot,\cdot,\cdot,\cdot)$ is a continuous linear form in $[D(A^\frac{1}{2})]^4$, given by
\begin{eqnarray}
\nonumber
\mathfrak{L}(\varphi^*,\psi^*, y^*,z^*) &= & \gamma_1\dual{A^{\beta_1}f^1}{\varphi^*}+\rho_1\dual{f^2}{\varphi^*}+\dfrac{\delta^2}{K}\dual{Af^3}{\psi^*}+\rho_2\dual{f^4}{\psi^*}+\dfrac{\delta\rho_5}{K}\dual{f^9}{\psi^*}\\
\label{46Ramos}
& & +\gamma_2\dual{A^{\beta_2}f^5}{y^*}+\rho_3\dual{f^6}{y^*}+\gamma_3\dual{A^{\beta_3}f^7}{z^*}+\rho_4\dual{f^8}{z^*}.
\end{eqnarray}
Since
\begin{equation*}
\mathfrak{B}((\varphi,\psi,y,z),(\varphi,\psi,y, z))=\kappa_1\|\varphi_x-\psi\|^2+\kappa_2\|y_x-z\|^2+b_1\|\psi_x\|^2+b_2\|z_x\|^2+\jmath\|y-\varphi\|^2,
\end{equation*}
the sesquilinear  form $\mathfrak{B}(\cdot,\cdot)$ is strongly coercive on $[D(A^\frac{1}{2})]^4$, and since \eqref{46Ramos} defines a continuous linear functional of $(\varphi^*,\psi^*, y^*, z^*)$, by Lax--Milgram's Theorem, problem \eqref{44Ramos} admits a unique solution $(\varphi,\psi,y,z)\in [D(A^\frac{1}{2})]^4$. By taking test functions in the form; $(\overline{\varphi},0,0,0), (0,\overline{\psi},0,0), (0,0,\overline{y},0)$ and $(0,0,0,\overline{z})$ with $\overline{\varphi},\overline{\psi},\overline{y}, \overline{z}\in \mathcal{D}(0,l)$ (espace of test functions), it is easy to see, that $(\varphi,\psi,y,z)$ satisfies equations \eqref{40Ramos}--\eqref{43Ramos} in the distributional sense. This also shows that $(\varphi,\psi,y,z)\in (\mathfrak{D}(A)\cap\mathfrak{D}(A^{\beta_1}))\times \mathfrak{D}(A)\times (\mathfrak{D}(A)\cap\mathfrak{D}(A^{\beta_2}))\times (\mathfrak{D}(A)\cap\mathfrak{D}(A^{\beta_3}))$ for all  $(\beta_1,\beta_2,\beta_3)\in [0,1]^3$, because
\begin{eqnarray}
\label{47Ramos}
\kappa_1 A\varphi& = &-\kappa_1\psi_x+\jmath (y-\varphi) +\gamma_1 A^{\beta_1}f^1+\rho_1 f^2,\\
\label{48Ramos}
b_1A\psi& = &\kappa_1(\varphi_x-\psi)+ \dfrac{\delta^2}{K}Af^3+\rho_2 f^4+ \dfrac{\delta\rho_5}{K} f^9,\\
\label{49Ramos}
\kappa_2 Ay& = & -\kappa_2 z_x-\jmath (y-\varphi) \gamma_2 A^{\beta_2}f^5+\rho_3f^6,
\\
\label{50Ramos}
b_2 Az& = &\kappa_2 (y_x-z) + \gamma_3 A^{\beta_3}f^7+\rho_4 f^8.
\end{eqnarray}
Since $-u=f^1\in D(A^\frac{1}{2} , -v=f^3\in D(A^\frac{1}{2}, -s=f^5\in D(A^\frac{1}{2}), -w=f^7\in D(A^\frac{1}{2}$, $A\theta=\frac{\delta}{K}Af^3+\frac{\rho_5}{K}f^9\in D(A^\frac{1}{2})$  we have proved that $(\varphi,u,\psi,v,y,s,z,w,\theta)^T$    belongs to $\mathfrak{D}(\mathbb{B}_2)$ and is a solutions of $-\mathbb{B}_2U=F$ and it is not difficult to prove that $\mathbb{B}_2^{-1}$ is a bounded operator $(\|U\|^2_{\mathbb{H}_2}=\|\mathbb{B}_2^{-1}F\|^2_{\mathbb{H}_2}\leq C\|F\|^2_{\mathbb{H}_2}$). Therefore, we conclude that $0\in \rho(\mathbb{B}_2)$, 
%Then, by the solvent equation, for $|\lambda| > 0$ small enough we have $R(i\lambda I-\mathbb{B}_2 ) = \mathbb{H}_2$.
 and this finish the proof of this Theorem \ref{TLiuZ}.

\subsection{Exponential Decay of System 02,  for $(\beta_1,\beta_2,\beta_3)\in [0,1]^3$}
\label{3.2}
In this section, we will study the asymptotic behavior of the semigroup  $S_2(t)=e^{t\mathbb{B}_2}$ of the system \eqref{Eq2.13}-\eqref{Eq2.17}. 
%%%%%%%  Remark 4
\begin{remark}\label{EqvExponential2}
 Note that to show the condition \eqref{Exponential} it is enough to show that: Let $\delta>0$. There exists a constant $C_\delta>0$ such that the solutions of the system \eqref{Eq2.13}-\eqref{Eq2.17} for $|\lambda|>\delta$,  satisfy the inequality 
 \begin{equation}\label{EqvExponencial2}
 \|U\|_{\mathbb{H}_2}\leq C_\delta\|F\|_{\mathbb{H}_2}\qquad {\rm for}\quad 0\leq\beta_1,\beta_2,\beta_3\leq 1.
 \end{equation}
 \end{remark}
%%%%%%%%%%%%%%%%%%%
%%%%%%%%%%%%%%%%%%%
In order to use Theorem \ref{LiuZExponential},  we will try to obtain some estimates  for:
 $$U=(\varphi,u,\psi, v,y, s,z,w, \theta)^T\in \mathfrak{D}(\mathbb{B}_2)\;{\rm and}\; F=(f^1, f^2, f^3, f^4, f^5, f^6, f^7, f^8, f^9)^T\in \mathbb{H}_2,$$  such that $(i\lambda I-\mathbb{B}_2)U=F$, where $\lambda\in \R$. This system, written in components,  reads
%%% % % % % % % % % %% % % % % % % % % % % % % % % % % %
%%%%%%%%%%%%%%%%%%%%%%%%%%%%%%%%%%%%%%%%%%%%%%%%%%
\begin{eqnarray}
i\lambda \varphi-u &=& f^1\quad {\rm in}\quad \mathfrak{D}(A^\frac{1}{2})\label{Pesp2-10}\\
i\lambda u+\dfrac{\kappa_1}{\rho_1}A\varphi+\dfrac{\kappa_1}{\rho_1}\psi_x-\dfrac{\jmath}{\rho_1}(y-\varphi)+\dfrac{\gamma_1}{\rho_1}A^{\beta_1}u &=& f^2\quad {\rm in}\quad \mathfrak{D}(A^0)\label{Pesp2-20}\\
i\lambda\psi-v &= & f^3\quad {\rm in}\quad \mathfrak{D}(A^\frac{1}{2})\label{Pesp2-30}\\
i\lambda v+\dfrac{b_1}{\rho_2}A\psi-\dfrac{\kappa_1}{\rho_2}(\varphi_x-\psi)-\dfrac{\delta}{\rho_2}A\theta &=& f^4\quad {\rm in}\quad \mathfrak{D}(A^0) \label{Pesp2-40}\\
i\lambda y-s &= &  f^5\quad {\rm in}\quad \mathfrak{D}(A^\frac{1}{2})\label{Pesp2-50}\\
i\lambda s+\dfrac{\kappa_2}{\rho_3}Ay+\dfrac{\kappa_2}{\rho_3}z_x+\dfrac{\jmath}{\rho_3}(y-\varphi)+\dfrac{\gamma_2}{\rho_3}A^{\beta_2} s &=& f^6\quad {\rm in}\quad \mathfrak{D}(A^0) \label{Pesp2-60}\\
i\lambda z-w &= & f^7\quad {\rm in}\quad \mathfrak{D}(A^\frac{1}{2})\label{Pesp2-70}\\
i\lambda w+\dfrac{b_2}{\rho_4}Az-\dfrac{\kappa_2}{\rho_4}(y_x-z)+\dfrac{\gamma_3}{\rho_4}A^{\beta_3}w &= & f^8\quad {\rm in}\quad \mathfrak{D}(A^0) \label{Pesp2-80}\\
i\lambda\theta+\dfrac{K}{\rho_5} A\theta+\dfrac{\delta}{\rho_5} Av &=& f^9 \quad {\rm in}\quad \mathfrak{D}(A^0). \label{Pesp2-90}
\end{eqnarray}
From \eqref{disipative2},   we have the first estimate
\begin{multline*}
|\gamma_1\|A^\frac{\beta_1}{2}u\|^2+\gamma_2\|A^\frac{\beta_2}{2}s\|^2+\gamma_3\|A^\frac{\beta_3}{2}w\|^2+ K\|A^\frac{1}{2}\theta\|^2|\\
 = |-{\rm Re}\dual{\mathbb{B}_2U}{U}|=|{\rm Re}\{ \dual{i\lambda U-F}{U}  \}|\\
\leq  |\dual{F}{U}|\leq \|F\|_{\mathbb{H}_2}\|\|U\|_{\mathbb{H}_2}.
\end{multline*}
Therefore
\begin{equation}\label{dis2-10}
\gamma_1\|A^\frac{\beta_1}{2}u\|^2+\gamma_2\|A^\frac{\beta_2}{2}s\|^2+\gamma_3\|A^\frac{\beta_3}{2}w\|^2+K\|A^\frac{1}{2}\theta\|^2 \leq  \|F\|_{\mathbb{H}_2}\|\|U\|_{\mathbb{H}_2}.
\end{equation}
%Next we will show the lemma that will be fundamental.
Next, we show some lemmas that will lead us to the proof of the main theorem of this section.\\
%%%%%%%%%%%%%%%%%%%%%%%%%%%%%
%%%%   Lemma  12    02 %%%%%%
%%%%%%%%%%%%%%%%%%%%%%%%%%%%%
\begin{lemma}\label{Lemma32}
Let $\delta>0$.  There exists $C_\delta>0$ such that the solutions of the system \eqref{Eq2.13}-\eqref{Eq2.17}  for  $(\beta_1,\beta_2,\beta_3)\in[0,1]^3$,  for $\varepsilon>0$, exists $C_\varepsilon>0$ independent of $\lambda$, such that
\begin{eqnarray}\label{Item01Lemma32}
\|v\|^2 & \leq &C_\varepsilon \|F\|_{\mathbb{H}_2}\|U\|_{\mathbb{H}_2}+\varepsilon\|U\|^2_{\mathbb{H}_2}.
\end{eqnarray}
\end{lemma}
%%%%%%%%  proof  %%%%
\begin{proof} 
Applying  the duality product between \eqref{Pesp2-90} and $A^{-1}v$ and using \eqref{Pesp2-40}, we have
\begin{eqnarray*}
\dfrac{\delta}{\rho_5}\|v\|^2=\dual{A^{-1}\theta}{i\lambda v}-\dfrac{K}{\rho_5}\dual{\theta}{v}+\dual{f^9}{A^{-1}v}\\
=\dual{A^{-1}\theta}{-\dfrac{b_1}{\rho_2}A\psi+\dfrac{\kappa_1}{\rho_2}(\varphi_x-\psi)+\dfrac{\delta}{\rho_2}A\theta-f^4}-\dfrac{K}{\rho_5}\dual{\theta}{v}+\dual{f^9}{A^{-1}v}\\
=-\dfrac{b_1}{\rho_2}\dual{\theta}{\psi}+\dfrac{\kappa_1}{\rho_2}\dual{A^{-1}\theta}{(\varphi_x-\psi)}+\dfrac{\delta}{\rho_2}\|\theta\|^2-\dual{\theta}{f^4}-\dfrac{K}{\rho_5}\dual{\theta}{v}+\dual{f^9}{A^{-1}v}
\end{eqnarray*}
Applying Cauchy-Schwarz inequality, we have
\begin{equation*}
\|v\|^2\leq C\{\|\theta\|\|\psi\|+\|A^{-1}\theta\|\|\varphi_x-\psi\|+\|\theta\|^2+\|\theta\|\|f^4\|+\|f^9\|\|A^{-1}v\|\}.
\end{equation*}
 Finally, applying Young inequality, continuous immersions: $\mathfrak{D}(A^\frac{1}{2})\hookrightarrow \mathfrak{D}(A^0)\hookrightarrow\mathfrak{D}(A^{-1})$, and from  estimative \eqref{dis2-10},   finish proof this lemma.
 
\end{proof}
%%%%%%%%%%%%%%%%%%%%%%%%%%%
%%%%%%%%%%%%%%%%%%%%%%%%%%%
%%%%   Lemma  13      %%%%%
%%%%%%%%%%%%%%%%%%%%%%%%%%%
\begin{lemma}\label{Lemma22}
Let $\delta>0$.  There exists $C_\delta>0$ such that the solutions of the system \eqref{Eq2.13}-\eqref{Eq2.17}  for $|\lambda| > \delta$ and $(\beta_1,\beta_2,\beta_3)\in [0,1]^3$,  satisfy
\begin{eqnarray}\label{Item01Lemma22}
(i)\quad 
|\lambda|\| y-\varphi\|^2  \leq C_\delta\|F\|_{\mathbb{H}_2}\|U\|_{\mathbb{H}_2},\\
\label{Item02Lemma22}
(ii)\quad \kappa_1\|\varphi_x-\psi\|^2+b_1\|A^\frac{1}{2}\psi\|^2\leq \varepsilon\|U\|^2_{\mathbb{H}_2}+C_\delta\|F\|_{\mathbb{H}_2}\|U\|_{\mathbb{H}_2},\\
\label{Item03Lemma22}
(iii)\quad \kappa_2\|y_x-z\|^2+b_2\|A^\frac{1}{2}z\|^2\leq C_\delta\|F\|_{\mathbb{H}_2}\|U\|_{\mathbb{H}_2}.
\end{eqnarray}
\end{lemma}
\begin{proof} 
We omit the proof of this lemma because it is completely similar to the proof of Lemma \ref{Lemma2} of system 1.
\end{proof}
%%%%%%%%%%%%%%%%%%%%%%%%%%%%%%%%%%%%%%%%%%%%%%%%%%%%%%%%%%%%%%%%%%%%%
%%%%%%%%%%%%%%%%%%%%%%%%%%%%%%%%%%%%%%%%%%%%%%%%%%%%%%%%%%%%%%%%%%%%%
%%%%%%%   Theorem  Decaimento Exponencial  Sistema 02   %%%%%%%%%%%%%
\begin{theorem}\label{TDExponential2}
The semigroup $S_2(t) = e^{t\mathbb{B}_2}$, is exponentially stable as long as the parameters $(\beta_1,\beta_2,\beta_3)\in [0,1]^ 3 $.
\end{theorem}
\begin{proof} 
Let's first check the condition \eqref{EqvExponencial}, which implies \eqref{Exponential}. Using  the Lemmas \ref{Lemma22},  \ref{Lemma32} and and applying in the sequence the estimates of \eqref{dis2-10}, we arrive at:
\begin{equation}\label{Eq012Exponential2}
\|U\|_{\mathbb{H}_2}^2\leq C_\delta \|F\|_{\mathbb{H}_2}\|U\|_{\mathbb{H}_2}\quad{\rm for}\quad 0\leq\beta_1,\beta_2,\beta_3\leq 1.
\end{equation}
Therefore the condition \eqref{Exponential} for $(\beta_1,\beta_2,\beta_3)\in [0,1]^3$ of Theorem \ref{LiuZExponential} is verified.  Next, we will announce a lemma of the condition \eqref{EImaginary}. The demonstration will be omitted, as it is completely similar to the one demonstrated for the first system.
%%%%%%%%%%%%%%%%%%%%%%%%%%%%%%%%%%
%%  Lemma 10:  Eixo Imaginário  %%
%%%%%%%%%%%%%%%%%%%%%%%%%%%%%%%%%%
\begin{lemma}\label{EImaginary2}
\label{iR2}
Let $\varrho(\mathbb{B}_2)$ be the resolvent set of operator
$\mathbb{B}_2$. Then
\begin{equation}
i\hspace{0.5pt}\mathbb{R}\subset\varrho(\mathbb{B}_2).
\end{equation}
\end{lemma}
%%%%%%%
\begin{proof}
The proof is similar to the proof of Lemma \ref{EImaginary}.
\end{proof}
%%%%%%%%%%%%%%%%%
Therefore,  the semigroup $S_2(t)=e^{t\mathbb{B}_2}$ is exponentially stable  for $(\beta_1,\beta_2,\beta_3)\in [0,1]^3$,  thus we finish the proof of this Theorem \ref{TDExponential2}. 
\end{proof}
%%%%%%%%%%%%%%%%%%%%%%%%%%%%%%%%%%%%%%%%%%%%%%%%%%%%%%%%%%%%%%%%%%%%%%%%%%%%%%%%%%%%%%
%%%%%%%%%%% Final del  decaimiento Exponencial  System 02   %%%%%%%%%%%%%%%%%%%%%%%%%%
%%%%%%%%%%%%%%%%%%%%%%%%%%%%%%%%%%%%%%%%%%%%%%%%%%%%%%%%%%%%%%%%%%%%%%%%%%%%%%%%%%%%%%
%%%%%%%%%%%%%%%%%%%%%%%%%%%%%%%%%
%%%%%    Theorem 4        %%%%%%%
%%%%%%%%%%%%%%%%%%%%%%%%%%%%%%%%%
\begin{theorem}[Lions' Interpolation]\label{Lions-Landau-Kolmogorov}  Let $\alpha<\beta<\gamma$. The there exists a constant $L=L(\alpha,\beta,\gamma)$ such that
\begin{equation}\label{ILLK}
\|A^\beta u\|\leq L\|A^\alpha u\|^\frac{\gamma-\beta}{\gamma-\alpha}\cdot \|A^\gamma u\|^\frac{\beta-\alpha}{\gamma-\alpha}
\end{equation}
for every $u\in\mathfrak{D}(A^\gamma)$.
\end{theorem}
\begin{proof}
See  Theorem  5.34 \cite{EN2000}.
\end{proof}
%%%%%%%%%%%%%%%%%%%%%%%%%%%%%%%%%
\subsection{Regularity  of the semigroup $S_2(t)=e^{t\mathbb{B}_2}$}
In this subsection, we will show that the semigroup $S_2(t)$ is analyticity for $(\beta_1,\beta_2,\beta_3)\in [\frac{1}{2},1]^3$ and  determination of Gevrey class for $(\beta_1,\beta_2,\beta_3)\in (0,1)^3$. But before we will show some preliminary lemmas.
\subsubsection{Analyticity: System 02}

The following theorem characterizes the analyticity of $S_2(t)$, see \cite{LiuZ}:
\begin{theorem}[see \cite{LiuZ}]
\label{LiuZAnaliticity}
    Let
$S_2(t)=e^{\mathbb{B}_2t}$ be
$C_0$-semigroup of contractions  
on a Hilbert space.  
Suppose that
    \begin{equation*}
    \rho(\mathbb{B}_2)
    \supseteq\{ i\lambda; \; \lambda\in \R \} 
     \equiv i\R
    \end{equation*}
     Then
$S_2(t)$ is analytic if and only if
    \begin{equation}\label{Analiticity}
     \limsup\limits_{|\lambda|\to
        \infty}
    \|\lambda(i\lambda I-\mathbb{B}_2)^{-1}
    \|_{\mathcal{L}(\mathbb{H}_2)}<\infty
    \end{equation}
    holds.
    
\end{theorem}
 
\begin{remark}\label{ObsEquivAnaly}  
\rm
To show the \eqref{Analiticity} condition,  
it suffices to show that, given
$\delta>0$ there exists a constant
$C_\delta > 0$ such that the solutions of 
\eqref{Pesp2-10}--\eqref{Pesp2-90}, 
for
$|\lambda|>\delta$ satisfy the inequality
\begin{equation}
\label{EquivAnaliticity}
 \|\lambda(i\lambda I-\mathbb{B}_2)^{-1}
F\|_{\mathbb{H}_2}^2\leq 
C_\delta\|F\|_{\mathbb{H}_2}\|U\|_{\mathbb{H}_2}\qquad \Longleftrightarrow \qquad 
|\lambda|\|U\|^2_{\mathbb{H}_2}\leq 
C_\delta\|F\|_{\mathbb{H}_2}\|U\|_{\mathbb{H}_2}.
\end{equation}

\end{remark}
%%%%%%%%%%%%%%%%%%%%%%%%%%%%%
%%%%%%    Lemma     19   %%%%
%%%%%%%%%%%%%%%%%%%%%%%%%%%%%
\begin{lemma}\label{Lemma62}
Let $\varepsilon>0$.  There exists $C_\varepsilon>0$ such that the solutions of the system \eqref{Eq2.13}-\eqref{Eq2.17},  satisfy
\begin{eqnarray}\label{EcA12v}
\|A^\frac{1}{2}v\|^2\leq C_\varepsilon\|F\|_{\mathbb{H}_2}\|U\|_{\mathbb{H}_2}.
\end{eqnarray}
\end{lemma}
\begin{proof}
Performing the duality product of \eqref{Pesp2-90} for $v$ and using \eqref{Pesp2-40}, we obtain
\begin{eqnarray*}
\dfrac{\delta}{\rho_5}\|A^\frac{1}{2}v\|^2 &= & \dual{\theta}{i\lambda v}-\dfrac{K}{\rho_5}\dual{A^\frac{1}{2}\theta}{A^\frac{1}{2}v}+\dual{f^9}{v}\\
& =&-\dfrac{b_1}{\rho_2}\dual{A^\frac{1}{2}\theta}{A^\frac{1}{2}\psi}+\dfrac{\kappa_1}{\rho_2}\dual{\theta}{\varphi_x-\psi}+\dfrac{\delta}{\rho_2}\|A^\frac{1}{2}\theta\|^2+\dual{\theta}{f^4}\\
& & -\dfrac{K}{\rho_5}\dual{A^\frac{1}{2}\theta}{A^\frac{1}{2}v}+\dual{f^9}{v},
\end{eqnarray*}
using estimative \eqref{dis2-10} and     applying Cauchy-Schwarz and Young inequalities, for $\varepsilon>0$, exists $C_\varepsilon>0$ independent of $\lambda$, such that
\begin{equation*}
\|A^\frac{1}{2}v\|^2\leq C_\varepsilon\|F\|_{\mathbb{H}_2}\|U\|_{\mathbb{H}_2}+\varepsilon \|A^\frac{1}{2}v\|^2.
\end{equation*}
\end{proof}
%%%%%%%%%%%%%%%%%%%%%%%%%%%
%%%%     Lemma  21   %%%%%%
%%%%%%%%%%%%%%%%%%%%%%%%%%%
\begin{lemma}\label{Lemma52}
Let $\delta>0$.  There exists $C_\delta>0$ such that the solutions of the system \eqref{Eq2.13}-\eqref{Eq2.17}  for $|\lambda| > \delta$,  satisfy
\begin{eqnarray}
\label{Item01Lemma52}
(i)\quad |\lambda|\|\theta\|^2& \leq &C_\delta \|F\|_{\mathbb{H}_2}\|U\|_{\mathbb{H}_2}\quad{\rm for}\quad 0\leq\beta_1,\beta_2,\beta_3\leq 1,\\
\label{Item02Lemma52}
(ii)\quad |\lambda|\|u\|^2 &\leq & C_\delta \|F\|_{\mathbb{H}_2}\|U\|_{\mathbb{H}_2}\quad {\rm for}\quad \dfrac{1}{2}\leq\beta_1\leq 1,\\
\label{Item03Lemma52}
(iii)\quad |\lambda|\|s\|^2 &\leq & C_\delta\|F\|_{\mathbb{H}_2}\|U\|_{\mathbb{H}_2}\quad {\rm for}\quad \dfrac{1}{2}\leq\beta_2\leq 1,\\
\label{Item04Lemma25}
(iv)\quad |\lambda|\|\varphi_x-\psi\|^2  & \leq & C_\delta \|F\|_{\mathbb{H}_2}\|U\|_{\mathbb{H}_2}\quad{\rm for}\quad \dfrac{1}{2}\leq \beta_1\leq 1,\\
\label{Item05Lemma52}
(v)\quad |\lambda|\|y_x-z\|^2 & \leq & C_\delta \|F\|_{\mathbb{H}_2}\|U\|_{\mathbb{H}_2}\quad{\rm for}\quad \dfrac{1}{2}\leq \beta_2\leq 1.
\end{eqnarray}
\end{lemma}
%%%%%%%%  proof  %%%%
%%%%%%%%%%%%%%%%%%%%%%%%%%%%%%%%%
\begin{proof} 
$(i)$ Taking the duality product between \eqref{Pesp2-90} and $\theta$, taking advantage of the self-adjointness of the powers of the operator $A$,   we arrive at:
 \begin{eqnarray*}
 i\lambda\|\theta\|^2 &= &-\dfrac{K}{\rho_5}\|A^\frac{1}{2}\theta\|^2-\dfrac{\delta}{\rho_5}\dual{A^\frac{1}{2}v}{A^\frac{1}{2}\theta}+\dual{f^9}{\theta}.
 \end{eqnarray*}
 Finally, taking imaginary part, applying Young inequality, estimates \eqref{dis-10} and \eqref{EcA12v} of Lemma \ref{Lemma62}, finish to proof this item.
\\
%%%%%%%%%%%%%%%%%%%%%%%%%%%%%%%%%%%%%%%%%%%%%%%%%%%%%%%%%%%%%%%%
{\bf  Proof:} $(ii)$ Taking the duality product between \eqref{Pesp2-20} and $\lambda A^{-\beta_1}u$, using \eqref{Pesp2-10} and \eqref{Pesp2-30}, we arrive at:
 \begin{eqnarray*}
 \dfrac{\gamma_1}{\rho_1}\lambda\|u\|^2\hspace*{-0.3cm} &= & -i|\lambda|^2\|A^{-\frac{\beta_1}{2}}u\|^2-\dfrac{\kappa_1}{\rho_1}\dual{\lambda\varphi}{A^{1-\beta_1}u}-\dfrac{\kappa_1}{\rho_1}\dual{\lambda\psi_x}{A^{-\beta_1}u}\\
 & & +\dfrac{\jmath}{\rho_1}\dual{\dfrac{\lambda}{\sqrt{|\lambda|}}(y-\varphi)}{\sqrt{|\lambda|}A^{-\beta_1}u}+\dual{f^2}{\lambda A^{-\beta_1}u}\\
 &=&\hspace*{-0.3cm} -i|\lambda|^2\|A^{-\frac{\beta_1}{2}}u\|^2+\dfrac{i\kappa_1}{\rho_1}\|A^\frac{1-\beta_1}{2}u\|^2+\dfrac{i\kappa_1}{\rho_1}\dual{A^\frac{1}{2}f^1}{A^{\frac{1}{2}-\beta_1}u}-\dfrac{i\kappa_1}{\rho_1}\dual{v}{A^{-\beta_1}u_x}\\
 & & +\dfrac{i\kappa_1}{\rho_1}\dual{f^3_x}{A^{-\beta_1}u} +\dfrac{\jmath}{\rho_1}\dual{\dfrac{\lambda}{\sqrt{|\lambda|}}(y-\varphi)}{\sqrt{|\lambda|}A^{-\beta_1}u}-\dfrac{i\kappa_1}{\rho_1}\dual{f^2}{A^{1-\beta_1}\varphi}\\
 & &-\dfrac{i\kappa_1}{\rho_1}\dual{f^2}{A^{-\beta_1}\psi_x}+\dfrac{i\jmath}{\rho_1}\dual{f^2}{A^{-\beta_1}(y-\varphi)}-\dfrac{i\gamma_1}{\rho_1}\dual{f^2}{u}+i\|A^{-\frac{\beta_1}{2}}f^2\|^2.
 \end{eqnarray*}
 Taking real part and applying Cauchy-Schwarz and Young inequalities, for $\varepsilon>0$, exists $C_\varepsilon>0$, such that
\begin{eqnarray*}
 |\lambda|\|u\|^2 & \leq &C\{\|A^\frac{1}{2}f^1\|\|A^{\frac{1}{2}-\beta_1}u\|+\|v\|\|A^{-\beta_1}u_x\|+\|f^3_x\|\|A^{-\beta_1}u\| \} +C_\varepsilon |\lambda|\|y-\varphi\|^2\\
 & & +\varepsilon |\lambda|\|A^{-\beta_1}u\|^2+C\{\|f^2\|\|A^{1-\beta_1}\varphi\|+\|f^2\|\|A^{-\beta_1}\psi_x\|\\
& &  +\|f^2\|\|A^{-\beta_1}(y-\varphi)\|+\|f^2\|\|u\|\},
 \end{eqnarray*}
as  from $\frac{1}{2}\leq \beta_1\leq 1$, we have $-\beta_1\leq 0$,  $1-\beta_1\leq \frac{1}{2}$, $\frac{1}{2}-\beta_1\leq 0$  and $-\frac{1}{2}\leq \frac{1-2\beta_1}{2}\leq 0$, then  $\mathfrak{D}(A^\frac{1}{2})\hookrightarrow \mathfrak{D}(A^{1-\beta_1})$  and  $\mathfrak{D}(A^0)\hookrightarrow \mathfrak{D}(A^{\frac{1}{2}-\beta_1})$,   furthermore, from the estimative  \eqref{Item01Lemma32} of  Lemma \ref{Lemma32}  and $\|A^{-\beta_1}u_x\|=\|A^\frac{1-2\beta_1}{2}u\|$, we finish to proof this item.\\
%%%%%%%%%%%%%%%%%%%%%%%%%%%%%%%%%%%%%%%%%%%%%%%%%%%%%%%%%%%%%%%%%%%%%%%%%%%%%%%%
{\bf Proof:} $(iii)$ Taking the duality product between \eqref{Pesp2-60} and $\lambda A^{-\beta_2}s$, using \eqref{Pesp2-50} and \eqref{Pesp2-60}, we arrive at:
 \begin{eqnarray*}
 \dfrac{\gamma_2}{\rho_3}\lambda\|s\|^2\hspace*{-0.3cm} &= &\hspace*{-0.3cm} -i|\lambda|^2\|A^{-\frac{\beta_2}{2}}s\|^2-\dfrac{\kappa_2}{\rho_3}\dual{\lambda y}{A^{1-\beta_2}s}-\dfrac{\kappa_2}{\rho_3}\dual{\lambda z_x}{A^{-\beta_2}s}\\
 & & -\dfrac{\jmath}{\rho_3}\dual{\dfrac{\lambda}{\sqrt{|\lambda|}}(y-\varphi)}{\sqrt{|\lambda|}A^{-\beta_2}s}+\dual{f^6}{\lambda A^{-\beta_2}s}\\
 &=& \hspace*{-0.3cm}-i|\lambda|^2\|A^{-\frac{\beta_2}{2}}s\|^2+\dfrac{i\kappa_2}{\rho_3}\|A^\frac{1-\beta_2}{2}s\|^2+\dfrac{i\kappa_2}{\rho_3}\dual{A^\frac{1}{2}f^5}{A^{\frac{1}{2}-\beta_2}s}-\dfrac{i\kappa_2}{\rho_3}\dual{w}{A^{-\beta_2}s_x}\\
 & & +\dfrac{i\kappa_2}{\rho_3}\dual{f^7_x}{A^{-\beta_1}s} +\dfrac{\jmath}{\rho_1}\dual{\dfrac{\lambda}{\sqrt{|\lambda|}}(y-\varphi)}{\sqrt{|\lambda|}A^{-\beta_2}s}-\dfrac{i\kappa_2}{\rho_3}\dual{f^6}{A^{1-\beta_2}y}\\
 & &-\dfrac{i\kappa_2}{\rho_3}\dual{f^6}{A^{-\beta_2}z_x}-\dfrac{i\jmath}{\rho_3}\dual{f^6}{A^{-\beta_2}(y-\varphi)}-\dfrac{i\gamma_2}{\rho_3}\dual{f^6}{s}+i\|A^{-\frac{\beta_2}{2}}f^6\|^2.
 \end{eqnarray*}
 Taking, real part, and applying Cauchy-Schwarz and Young inequalities, for $\varepsilon>0$, exists $C_\varepsilon>0$, such that
\begin{eqnarray*}
 |\lambda|\|s\|^2 & \leq &C\{\|A^\frac{1}{2}f^5\|\|A^{\frac{1}{2}-\beta_2}s\|+\|w\|\|A^{-\beta_2}s_x\|+\|f^7_x\|\|A^{-\beta_2}s\| \} +C_\varepsilon |\lambda|\|y-\varphi\|^2\\
 & & +\varepsilon |\lambda|\|A^{-\beta_2}s\|^2+C\{\|f^6\|\|A^{1-\beta_2}y\|+\|f^6\|\|A^{-\beta_2}z_x\|\\
& &  +\|f^6\|\|A^{-\beta_2}(y-\varphi)\|+\|f^6\|\|s\|\},
 \end{eqnarray*}
%%%%%%%%
as  from $\frac{1}{2}\leq \beta_2\leq 1$, we have $-\beta_2\leq 0$,  $1-\beta_2\leq \frac{1}{2}$, $\frac{1}{2}-\beta_2\leq 0$  and $-\frac{1}{2}\leq \frac{1-2\beta_2}{2}\leq 0$, then  $\mathfrak{D}(A^\frac{1}{2})\hookrightarrow \mathfrak{D}(A^{1-\beta_2})$  and  $\mathfrak{D}(A^0)\hookrightarrow \mathfrak{D}(A^{\frac{1}{2}-\beta_2})$,   furthermore, from the estimative  \eqref{Item01Lemma32} of  Lemma \ref{Lemma32}  and $\|A^{-\beta_2}z_x\|=\|A^\frac{1-2\beta_2}{2}z\|$, we finish to proof this item.
\\
%%%%%%%%%%%%%%%%%%%%%%%%%%%%%%%%%%%%%%%%%%%%%%%%%%%%
{\bf Proof:} $(iv)$ 
%\begin{comment}
 From \eqref{Pesp2-10}, we have $i\lambda\varphi_x-u_x=f^1_x$, subtracting from this result the equation \eqref{Pesp2-30}, we have
\begin{equation}\label{Eq03Lemma052}
i\lambda(\varphi_x-\psi)-(u_x-v)=f^1_x-f^3,
\end{equation} 
taking the duality product between \eqref{Eq03Lemma052} and $\varphi_x-\psi$ and using \eqref{Pesp2-20}, we arrive at:
\begin{eqnarray*}
i\lambda\|\varphi_x-\psi\|^2 &= &\dual{(u_x-v)}{\varphi_x-\psi}+\dual{f_x^1}{(\varphi_x-\psi)}-\dual{f^3}{(\varphi_x-\psi)}\\
&=&-\dual{u}{(\varphi_x-\psi)_x}-\dual{v}{\varphi_x-\psi}+\dual{f_x^1}{(\varphi_x-\psi)}-\dual{f^3}{(\varphi_x-\psi)}\\
&= &\dfrac{i\rho_1}{\kappa_1}\lambda\|u\|^2-\dfrac{\jmath}{\kappa_1}\dual{u}{(y-\varphi)}-\dfrac{\gamma_1}{\kappa_1}\|A^\frac{\beta_1}{2}u\|^2+\dfrac{\rho_1}{\kappa_1}\dual{u}{f^2} \\
& & -\dual{v}{(\varphi_x-\psi)}+\dual{f_x^1}{(\varphi_x-\psi)}-\dual{f^3}{(\varphi_x-\psi)},
\end{eqnarray*}
taking imaginary part and applying Cauchy-Schwarz and Young inequalities, and using \eqref{Eq012Exponential2},  we have
%\end{comment}
\begin{equation}\label{Eq04Lemma052}
|\lambda|\|\varphi_x-\psi\|^2\leq C\{|\lambda|\|u\|^2+\|F\|_{\mathbb{H}_2}\|U\|_{\mathbb{H}_2}\}\quad{\rm for}\quad 0\leq\beta_1,\beta_2, \beta_3\leq 1.
\end{equation}
Using  \eqref{Item02Lemma52} (item $(ii)$ this lemma) we finish proof of this item.
\\
{\bf Proof:} $(v)$ On the other hand, similarly
 %%%%%%%%%%%%%%%%
 from \eqref{Pesp2-50}, we have $i\lambda y_x-s_x=f^5_x$, subtracting from this result the equation \eqref{Pesp2-70}, we have  
\begin{equation}\label{Eq05Lemma052}
i\lambda(y_x-z)-(s_x-w)=f^5_x-f^7.
\end{equation} 
Taking the duality product between \eqref{Eq05Lemma052} and $y_x-z$ and using \eqref{Pesp2-40}, we arrive at:
\begin{eqnarray*}
i\lambda\|y_x-z\|^2 &= &\dual{(s_x-w)}{y_x-z}+\dual{f_x^5}{y_x-z}-\dual{f^7}{y_x-z}\\
&=&-\dual{s}{(y_x-z)_x}-\dual{w}{(y_x-z)}+\dual{f_x^5}{y_x}-\dual{f^5_x}{z}\\
& & -\dual{f^7}{y_x}+\dual{f^7}{z}\\
&= &\dfrac{i\rho_3\lambda}{\kappa_2}\|s\|^2-\dfrac{\jmath}{\kappa_2}\dual{s}{(y-\varphi)}-\dfrac{\gamma_2}{\kappa_2}\|A^\frac{\tau_2}{2}s\|^2+\dfrac{\rho_3}{\kappa_2}\dual{s}{f^6} \\
& & -\dual{w}{(y_x-z)}+\dual{f_x^5}{y_x}-\dual{f^5_x}{z}-\dual{f^7}{y_x}+\dual{f^7}{z}.
\end{eqnarray*}
Taking imaginary part and applying Cauchy-Schwarz and Young inequalities, we have
 \begin{equation}\label{Eq06Lemma052}
|\lambda|\|y_x-z\|^2\leq C\{ |\lambda|\|s\|^2+\|F\|_\mathbb{H}\|U\|_\mathbb{H}\}\quad{\rm for}\quad 0\leq\beta_1,\beta_2\beta_3\leq 1.
\end{equation}
Using  \eqref{Item03Lemma52} (item $(iii)$ this lemma) we finish proof of this lemma
.
%%%%%%%%%%%%%%%%%%%%
\end{proof}
%%%%%%%%%%%%%%%%%%%%%%%%%%%%%%%%%%%%%%%%%%%%%%%%%%%%%%%%%%%%%%%%%%%%%%%%%%%%%%%%%%
%%%%%%%%%%%%%%%%%%%%%%%%%%%
%%%       Lema 21    %%%%%%
%%%%%%%%%%%%%%%%%%%%%%%%%%%
\begin{lemma}\label{Lemma72}
Let $\delta>0$.  There exists $C_\delta>0$ such that the solutions of the system \eqref{Eq2.13}-\eqref{Eq2.17}  for $|\lambda| > \delta$,  satisfy
\begin{eqnarray}
\label{EcLambdav}
|\lambda|\|v\|^2\leq C_\delta\|F\|_{\mathbb{H}_2}\|U\|_{\mathbb{H}_2}\quad{\rm for}\quad 0\leq\beta_1,\beta_2,\beta_3\leq 1.
\end{eqnarray}
\end{lemma}
\begin{proof}
Performing the duality product between equation \eqref{Pesp2-90} and $\lambda A^{-1}v$, and using \eqref{Pesp2-30} and \eqref{Pesp2-40}, we obtain
\begin{eqnarray*}
\dfrac{\delta}{\rho_5}\lambda\|v\|^2 &= & \dual{\lambda A^{-1}\theta}{i\lambda v}-\dfrac{K}{\rho_5}\dual{\theta}{\lambda v}+\dual{f^9}{\lambda A^{-1}v}\\
&=& -i\dfrac{b_1}{\rho_2}\dual{\theta}{v}-i\dfrac{b_1}{\rho_2}\dual{\theta}{f^3}+\dfrac{i\kappa_1K}{\rho_2\rho_5}\dual{\theta}{\varphi_x-\psi}+\dfrac{i\kappa_1\delta}{\rho_2\rho_5}\dual{v}{\varphi_x-\psi}\\
& & -\dfrac{i\kappa_1}{\rho_2}\dual{A^{-1}f^9}{\varphi_x-\psi}+\dfrac{\delta}{\rho_2}\lambda\|\theta\|^2+\dfrac{iK}{\rho_5}\dual{\theta}{f^4}+\dfrac{i\delta}{\rho_5}\dual{v}{f^4}\\
& & -i\dual{A^{-1}f^9}{f^4}-\dfrac{K}{\rho_5}\dual{\sqrt{|\lambda|}\theta}{\dfrac{\lambda}{\sqrt{|\lambda|}} v}-\dfrac{ib_1}{\rho_2}\dual{f^9}{\psi}\\
& & +\dfrac{i\kappa_1}{\rho_2}\dual{A^{-1}f^9}{\varphi_x-\psi}+\dfrac{i\delta}{\rho_2}\dual{f^9}{\theta}+i\dual{A^{-1}f^9}{f^4}.
\end{eqnarray*}
Applying Cauchy-Schwarz and Young inequalities, for $\varepsilon>0$, exists $C_\varepsilon>0$, such that
\begin{eqnarray*}
|\lambda|\|v\|^2 &\leq & C\{\|\theta\|\|v\|+\|\theta\|\|f^3\|+\|\theta\|\|\varphi_x-\psi\|+\|v\|\|\varphi_x-\psi\|+|\lambda|\|\theta\|^2\\
& & +\|\theta\|\|f^4\|+\|v\|\|f^4\|+\|f^9\|\|\psi\|+\|f^9\|\|\theta\|\}+C_\varepsilon|\lambda|\|\theta\|^2 +\varepsilon|\lambda|\|v\|^2.
\end{eqnarray*}
Finally, from estimates \eqref{EqvExponencial2}, \eqref{Item01Lemma52}, finish proof this lemma.
\end{proof}
%%%%%%%%%%%%%%%%%%%%%%%%%%%%%%%%%%%%%%%%%%%%%%%%%%%%%%%%%
%%%%%%%%%%%%%%%%%%%%%%%%%%%
%%%%   Lemma  22    %%%%%%%
%%%%%%%%%%%%%%%%%%%%%%%%%%%
\begin{lemma}\label{Lemma82}
Let $\delta>0$.  There exists $C_\delta>0$ such that the solutions of the system \eqref{Eq2.13}-\eqref{Eq2.17}  for $|\lambda| > \delta$,  satisfy
\begin{eqnarray}\label{Item01Lemma82}
(i)\; |\lambda|\|A^\frac{1}{2}\psi\|^2& \leq &  C_\delta\|F\|_{\mathbb{H}_2}\|U\|_{\mathbb{H}_2}\quad{\rm for}\quad 0\leq\beta_1,\beta_2,\beta_3\leq 1,\\
\label{Item02Lemma82}
(ii)\; |\lambda|\|w\|^2 &\leq &    C_\delta\|F\|_{\mathbb{H}_2}\|U\|_{\mathbb{H}_2}\quad{\rm for}\quad \dfrac{1}{2}\leq\beta_2,\beta_3\leq 1,\\
\label{Item03Lemma82}
(iii)\quad |\lambda|\|A^\frac{1}{2}z\|^2 &\leq &  C_\delta\|F\|_{\mathbb{H}_2}\|U\|_{\mathbb{H}_2}\quad{\rm for}\quad \dfrac{1}{2} \leq\beta_2,\beta_3\leq 1.
\end{eqnarray}
\end{lemma}
%%%%%%%%  proof  %%%%
%%%%%%%%%%%%%%%%%%%%%%%%%%%%%%%%%
\begin{proof}\!$(i)$  
 Performing the duality product between equation \eqref{Pesp2-40} and $\frac{i\rho_2}{\kappa_1}\lambda\psi$, we have
\begin{eqnarray}
\nonumber
\dfrac{ib_1}{\kappa_1}\lambda\|A^\frac{1}{2}\psi\|^2\hspace*{-0.25cm}& =\hspace*{-0.25cm} &i\dual{\sqrt{|\lambda|}(\varphi_x-\psi)}{\dfrac{\lambda}{\sqrt{|\lambda|}} \psi} +\dfrac{i\rho_2}{\kappa_1}\lambda\|v\|^2+\dfrac{\rho_2}{\kappa_1}\dual{i\lambda v}{f^3}-\dfrac{\delta}{\kappa_1}\dual{A^\frac{1}{2}\theta}{A^\frac{1}{2}v}\\
\label{Ec000B2}
& & -\dfrac{\delta}{\kappa_1}\dual{A^\frac{1}{2}\theta}{A^\frac{1}{2}f^3}-\dfrac{\rho_2}{\kappa_1}\dual{f^4}{v}-\dfrac{\rho_2}{\kappa_1}\dual{f^4}{f^3}
\end{eqnarray}
 as, of \eqref{Pesp2-40}, we have
\begin{equation}\label{Ec000C2}
\dfrac{\rho_2}{\kappa_1}\dual{i\lambda v}{f^3}=-\dfrac{b_1}{\kappa_1}\dual{A^\frac{1}{2}\psi}{A^\frac{1}{2}f^3}+\dual{\varphi_x-\psi}{f^3}+\dfrac{\delta}{\kappa_1}\dual{A^\frac{1}{2}\theta}{A^\frac{1}{2}f^3}+\dfrac{\rho_2}{\kappa_1}\dual{f^4}{f^3},
\end{equation}
using \eqref{Ec000C2} in \eqref{Ec000B2}, we have
\begin{eqnarray}
\nonumber
\dfrac{ib_1}{\kappa_1}\lambda\|A^\frac{1}{2}\psi\|^2\hspace*{-0.25cm} & = &\hspace*{-0.25cm} \dfrac{i\rho_2}{\kappa_1}\lambda\|v\|^2+ i\dual{\sqrt{|\lambda|}(\varphi_x-\psi)}{\dfrac{\lambda}{\sqrt{|\lambda|}} \psi} -\dfrac{\delta}{\kappa_1}\dual{A^\frac{1}{2}\theta}{A^\frac{1}{2}v}-\dfrac{\rho_2}{\kappa_1}\dual{f^4}{v}\\
\label{Ec000D2}
& & -\dfrac{b_1}{\kappa_1}\dual{A^\frac{1}{2}\psi}{A^\frac{1}{2}f^3}+\dual{\varphi_x-\psi}{f^3}.
\end{eqnarray}
Applying Cauchy-Schwarz and Young  inequalities, for $\varepsilon>0$, exists $C_\varepsilon>0$ independent of $\lambda$,  such that
\begin{multline*}
|\lambda|\|A^\frac{1}{2}\psi\|^2\leq C_\varepsilon|\lambda|\|\varphi_x-\psi\|^2+\varepsilon|\lambda|\|\psi\|^2+C\{\|A^\frac{1}{2}\theta\|^2+\|A^\frac{1}{2}v\|+|\lambda|\|v\|^2+\|f^4\|\|v\|\\
+\|A^\frac{1}{2}\psi\|\|A^\frac{1}{2}f^3\|+\|\varphi_x\|\|f^3\|+\|\psi\|\|f^3\|\},
\end{multline*}
finally, from $\mathfrak{D}(A^0)\hookrightarrow \mathfrak{D}(A^\frac{1}{2})$, estimates  \eqref{EqvExponencial2}, \eqref{dis2-10}, \eqref{EcA12v} Lemma \ref{Lemma62} and \eqref{EcLambdav} Lemma \ref{Lemma72}, we finish to proof this item.
%%%%%%%%%%%%%%
\\
{\bf Proof:} $(ii)$ Performing the duality product between equation  \eqref{Pesp2-80} and  $\lambda A^{-\beta_3}w$,  and using \eqref{Pesp2-70}, we obtain
\begin{eqnarray}
\nonumber
\dfrac{\gamma_3}{\rho_4}\lambda\|w\|^2  &= &-i|\lambda|^2\|A^{-\frac{\beta_3}{2}}w\|^2-\dfrac{b_2}{\rho_4}\dual{\lambda z}{A^{1-\beta_3}w}+\dfrac{\kappa_2}{\rho_4}\dual{\dfrac{\lambda}{\sqrt{|\lambda|}}(y_x-z)}{\sqrt{|\lambda|}A^{-\beta_3}w}\\
\nonumber
& & +\dual{f^8}{\lambda A^{-\beta_3}w}\\
\label{Ec0052}
&=&  -i|\lambda|^2\|A^{-\frac{\beta_3}{2}}w\|^2+\dfrac{ib_2}{\rho_4}\|A^\frac{1-\beta_3}{2}w\|^2+\dfrac{ib_2}{\rho_4}\dual{A^\frac{1}{2}f^7}{A^{\frac{1}{2}-\beta_3}w}\\
\nonumber
& &  +\dfrac{\kappa_2}{\rho_4}\dual{\dfrac{\lambda}{\sqrt{|\lambda|}}(y_x-z)}{\sqrt{|\lambda|}A^{-\beta_3}w}-\dfrac{ib_2}{\rho_4}\dual{f^8}{A^{1-\beta_3}z}\\
& &+\dfrac{i\kappa_2}{\rho_4}\dual{f^8}{A^{-\beta_3}(y_x-z)} -\dfrac{i\gamma_3}{\rho_4}\dual{f^8}{w}+i\|f^8\|^2.
\end{eqnarray}
 Taking, real part, and applying Cauchy-Schwarz and Young inequalities, for $\varepsilon>0$, exists $C_\varepsilon>0$, such that
 \begin{eqnarray*}
 |\lambda|\|w\|^2 &\leq & C_\varepsilon|\lambda|\|y_x-z\|^2+\varepsilon|\lambda|\|A^{-\beta_3}w\|^2+C\{\|A^\frac{1}{2}f^7\|\|A^{\frac{1}{2}-\beta_3}w\|\\
 & & +\|f^8\|\|A^{1-\beta_3}z\|+\|f^8\|\|A^{-\beta_3}(y_x-z)\|+\|f^8\|\|w\|.
 \end{eqnarray*}
as form $\frac{1}{2}\leq\beta_3\leq 1$, we have: $\frac{1}{2}-\beta_3\leq 0$ and  $1-\beta_3\leq\frac{1}{2}$, then $\mathfrak{D}(A^0)\hookrightarrow \mathfrak{D}(A^{\frac{1}{2}-\beta_3})$  and  $\mathfrak{D}(A^\frac{1}{2})\hookrightarrow \mathfrak{D}(A^{1-\beta_3})$. Finally applying estimative \eqref{Item05Lemma52} of Lemma \ref{Lemma52}, finish proof this item.\\
%%%%%%%%%%%%%%%%%%%%%%%%%%%%%%%%%%%%%%%%%%%%%%%%%%%%%%%%%%%%%%%%%%%
{\bf Proof:} $(iii)$   Performing the duality product between equation \eqref{Pesp2-80} and $w$  and using \eqref{Pesp2-70}, we have
\begin{eqnarray*}
i\lambda\|w\|^2-i\dfrac{b_2}{\rho_4}\lambda\|A^\frac{1}{2}z\|^2-\dfrac{b_2}{\rho_4}\dual{A^\frac{1}{2}z}{A^\frac{1}{2}f^7}-\dfrac{\kappa_2}{\rho_4}\dual{y_x-z}{w}+\dfrac{\gamma_3}{\rho_4}\|A^\frac{\beta_3}{2}w\|^2=\dual{f^8}{w},
\end{eqnarray*}
Taking imaginary part, and applying Cauchy-Schwarz  and Young inequalities, we obtain
\begin{multline}\label{Ec-wZ}
|\lambda|\|A^\frac{1}{2}z\|^2\leq C\{|\lambda|\|w\|^2+\|A^\frac{1}{2}z\|\|A^\frac{1}{2}f^7\|+\|y_x-z\|^2+\|w\|^2+\|f^8\|\|w\|\}\\
\leq C_\delta\{\|F\|_{\mathbb{H}_2}\|U\|_{\mathbb{H}_2}+|\lambda|\|w\|^2\}\quad \text{for}\quad 0\leq \beta_1,\beta_2,\beta_3\leq 1.
\end{multline}
Finally, applying of item (ii) this Lemma  and \eqref{EqvExponencial2}, finish to proof this item.
\end{proof}
%%%%%%%%%%%%%%%%%%%%%%%%%%%%%%%%%%%%%%%%%%%%%%%%%%%%%%%%%%%%%%%%%%%%%%%%%%
%% Final dos Lemas e Prova do Teorema  Principal de Analiticidad  %%%%%%%%
%%%%%%%%%%%%%%%%%%%%%%%%%%%%%%%%%%%%%%%%%%%%%%%%%%%%%%%%%%%%%%%%%%%%%%%%%%
\begin{theorem}\label{AnaliticidadeS2}
The semigroup $S_2(t)=e^{t\mathbb{B}_2}$  is analytic for $(\beta_1,\beta_2,\beta_3)\in [\frac{1}{2}, 1]^3$.
\end{theorem}
 \begin{proof}
 
 %%%%%%%%%%%
\begin{center}
\tdplotsetmaincoords{80}{-35}
\begin{tikzpicture}[tdplot_main_coords, scale=4.5,]
   \coordinate(A) at (0.5,0.5,0.5);
    \coordinate(B) at (1,0.5,0.5);
    \coordinate(C) at (1,0.5,1);
    \coordinate(D) at (0.5,0.5,1);
    \coordinate(E) at (0.5,1,0.5);
    \coordinate(F) at (1,1,0.5);
    \coordinate(G) at (1,1,1);
    \coordinate(H) at (0.5,1,1);
    \filldraw[black!10, fill=blue!20](E)--(F)--(G)--(H);
    \filldraw[black!10, fill=blue!20](A)--(B)--(F)--(E);
    \filldraw[black!10, fill=blue!20](C)--(G)--(H)--(D);
    \filldraw[black!0, fill=blue!20](A)--(B)--(C)--(D)--(A);
    \draw [dashed] (A)--(B);
    \draw (C)--(D);
    \draw [dashed] (D)--(0.5,0.5,0);
    \draw [dashed] (C)--(1,0.5,0);
    \draw [dashed] (E)--(F);
    \draw [dashed] (0.5,1,0)--(H);
    \draw (H)--(G);
    \draw [dashed] (1,1,0)--(1,0,0);
    \draw [dashed] (0.5,1,0)--(0.5,0,0);
    \draw [dashed] (0,0.5,0)--(1,0.5,0);
    \draw [dashed] (1,1,0)--(G);
    \draw [dashed] (A)--(E);
    \draw [dashed] (B)--(F);
    \draw (C)--(G);
    \draw (D)--(H);
    \draw (D)--(0,0.5,1)--(0,1,1)--(H);
    \draw (0,0.5,0)--(0,0.5,1);
    \draw (0,1,0)--(0,1,1);
    \draw (0,0.5,0.5)--(0,1,0.5);
    \draw [dashed] (0,1,0)--(1,1,0);
    \draw [dashed] (0,1,0.5)--(E);
    \draw [dashed] (A)--(0,0.5,0.5);
    \draw (0,0.5,0.5)--(0,0,0.5);
    \draw [dashed] (0.5,0.5,0.5)--(0.5,0,0.5);
    \draw [dashed] (1,0.5,0.5)--(1,0,0.5);
    \draw (1,0.5,1)--(1,0,1);
    \draw (0.5,0.5,1)--(0.5,0,1);
    \draw (0,0.5,1)--(0,0,1);
    \draw (0,0,0.5)--(1,0,0.5);
    \draw (0,0,1)--(1,0,1);
    \draw (0.5,0,0)--(0.5,0,1);
    \draw (1,0,0)--(1,0,1);
    \draw[->, black!60] (0, 0,0) -- (1.2, 0,0);
    \draw[->, black!60] (0, 0,0) -- (0, 1.2,0);
    \draw[->, black!60] (0, 0,0) -- (0, 0,1.4);
    \draw node at (1.25, 0,0) {\Large $\beta_1$};
    \draw node at (0, 1.25,0) {\Large $\beta_2$};
    \draw node at (0, 0,1.45) {\Large $\beta_3$};
    \draw node at (0, 0,-0.05) {\large $0$};
    \draw node at (1, 0,-0.05) {\large $1$};
    \draw node at (0, 1,-0.05) {\large $1$};
    \draw node at (0.5, 0,-0.06) { $\frac{1}{2}$};
    \draw node at (0, 0.5,-0.06) { $\frac{1}{2}$};
    \draw node at (-0.07, -0.05,0.5) { $\frac{1}{2}$};
    \draw node at (-0.1, -0.09,1) {\large $1$};
    \draw[fill=black](0.5,0.5,0.5) circle (0.3pt);
    \draw[fill=black](1,1,1) circle (0.3pt);
    \draw[fill=black](0.5,1,1) circle (0.3pt);
    \draw[fill=black](0.5,0.5,1) circle (0.3pt);
    \draw[fill=black](1,0.5,1) circle (0.3pt);
    \draw[fill=black](1,0.5,0.5) circle (0.3pt);
    \draw[fill=black](0.5,1,0.5) circle (0.3pt);
    \draw[fill=black](1,1,0.5) circle (0.3pt);
    
\end{tikzpicture}
\end{center}
\begin{center}
{\bf FIG. 01:} Region $R_{A2}$ of Analyticity de $S_2(t)=e^{t\mathcal{B}_2}$
\end{center}

 From Lemma \ref{EImaginary2}, \eqref{iR2} is verified.  Let $\delta>0$, there exists a constant $C_\delta>0$ such that the solutions of the system \eqref{Eq2.10}-\eqref{Eq2.17} for $|\lambda|>\delta$,  satisfy the inequality 
 \begin{equation}\label{EquiAnalyticity2} 
 |\lambda|\|U\|^2_{\mathbb{H}_2}\leq C_\delta\|F\|_{\mathbb{H}_2}\|U\|_{\mathbb{H}_2}.
 \end{equation}
Finally,   considering $(\beta_1,\beta_2,\beta_3)\in [\frac{1}{2}, 1]^3$  and  using \eqref{Item01Lemma22}  (item $(i)$ the Lemmas \ref{Lemma22}), and  Lemmas:  \ref{Lemma52}, \ref{Lemma72} and \ref{Lemma82}, we finish the proof of this theorem.

\end{proof}

 %%%%%%  FINAL DA PROVA DO THEOREM
%%%%%%%%%%%%%%%%
\subsubsection{Determination of Gevrey Classes: System 02}
\label{3.3.2}
Before exposing our results, it is useful to recall the next definition and result  presented in \cite{SCRT1990} (adapted from
\cite{TaylorM}, Theorem 4, p. 153]).

\begin{definition}\label{Def1.1Tebou} Let $t_0\geq 0$ be a real number. A strongly continuous semigroup $S(t)$, defined on a Banach space $ \mathbb{H}$, is of Gevrey class $s > 1$ for $t > t_0$, if $S(t)$ is infinitely differentiable for $t > t_0$, and for every compact set $K \subset (t_0,\infty)$ and each $\mu > 0$, there exists a constant $ C = C(\mu, K) > 0$ such that
    \begin{equation}\label{DesigDef1.1}
    ||S^{(n)}(t)||_{\mathcal{L}( \mathbb{H})} \leq  C\mu ^n(n!)^s,  \text{ for all } \quad t \in K, n = 0,1,2...
    \end{equation}
\end{definition}
\begin{theorem}[\cite{TaylorM}]\label{Theorem1.2Tebon}
    Let $S(t)$  be a strongly continuous and bounded semigroup on a Hilbert space $ \mathbb{H}$. Suppose that the infinitesimal generator $\mathbb{B}$ of the semigroup $S(t)$ satisfies the following estimate, for some $0 < \Psi < 1$:
    \begin{equation}\label{Eq1.5Tebon2020}
    \lim\limits_{|\lambda|\to\infty} \sup |\lambda |^\Psi ||(i\lambda I-\mathbb{B})^{-1}||_{\mathcal{L}( \mathbb{H})} < \infty.
    \end{equation}
    Then $S(t)$  is of Gevrey  class  $s$   for $t>0$,  for every   $s >\dfrac{1}{\Psi}$.
\end{theorem}
%%%%%%%%%%%%%%%%%%%%%%%%%%%%%%%%%%%%%%%%%%%%%%%%%%%%%%%%%%%%%%%%%%%%%%%%%%%%%%%%%%%%%%%%%%
%%%%%%%%%%%%%%%%%%%%%%%%%%%%%%%%%%%%%%%%%%%%%%%%%%%%%%%%%%%%
%%%%%%%%%%%%%%%%%%%%%%%%%%%%%%%%%%%%%%%%%%%%%%%%%%%%%%%%%%%%
Our main result in this subsection is as follows:
\begin{theorem} \label{TGevreyC}
Let  $S_2(t)=e^{t\mathbb{B}_2}$  strongly continuos-semigroups of contractions on the Hilbert space $ \mathbb{H}_2$, the semigroups $S_2(t)$ is of Gevrey class $s$,  for every $s> \frac{1+\phi}{2\phi}$,  such that,  we have the resolvent estimative:
  \begin{equation}\label{Gevrey01}
 \limsup_{|\lambda|\to\infty} |\lambda |^\frac{2\phi}{1+\phi} ||(i\lambda I-\mathbb{B}_2)^{-1}||_{\mathcal{L}( \mathbb{H}_2)} <\infty, 
    \end{equation}
where,  
\begin{equation}\label{phi}
\phi:=\min\limits_{(\beta_1,\beta_2,\beta_3)\in (0,1)^3}\{\beta_1,\beta_2, \beta_3\}.
\end{equation}
\end{theorem}
%%%%%%%
\begin{proof}
Notice that, for $\phi$ defined in \eqref{phi}, we have $0<(2\phi)/(\phi+1)<1$. Next we will estimate:   $|\lambda|^\frac{2\beta_1}{1+\beta_1}\|u\|^2, \quad|\lambda|^\frac{2\beta_2}{1+\beta_2}\|s\|^2$  and  $|\lambda|^\frac{2\beta_3}{1+\beta_3}\|w\|^2$.  \\
%%%%%%%%%%%%%%%
%%%%%%%%%%%%%%%%%%%%%%%%%%%%%%%%%%%%%%%
%%%%%%%%%%%  Estimativa de $|\lambda|\|w\|^2$    %%%%%%%%%%
%%%%%%%%%%%%%%%%%%%%%%%%%%%%%%%%%%%%%%
%%
{\bf  Let's start by estimating the term $|\lambda|^\frac{2\beta_1}{1+\beta_1}\|u\|$:}  It is  assume that   $|\lambda|>1$,  some ideas could be borrowed  from \cite{LiuR95}.  Set $u=u_1+u_2$, where $u_1\in \mathfrak{D}(A)$ and $u_2\in \mathfrak{D}(A^0)$, with 
\begin{equation}\label{Eq110AnalyRR}
i\lambda u_1+A u_1=f^2, 
\hspace{2cm} i\lambda u_2=-\dfrac{\kappa_1}{\rho_1}A\varphi-\dfrac{\kappa_1}{\rho_1}\psi_x+\dfrac{\jmath}{\rho_1}(y-\varphi)-\dfrac{\gamma_1}{\rho_1}A^{\beta_1}u+Au_1.
\end{equation} 
Firstly,  applying in the product duality  the first equation in \eqref{Eq110AnalyRR} by $u_1$, then  by $Au_1$    and recalling that the operator $A$  is
self-adjoint, resulting in 
\begin{equation}\label{Eq112AnalyRR}
|\lambda|\|u_1\|+|\lambda|^\frac{1}{2}\|A^\frac{1}{2}u_1\|+\|Au_1\|\leq C\|F\|_{\mathbb{H}_2}.
\end{equation}
Applying the $A^{-\frac{1}{2}}$ operator on the second equation of \eqref{Eq110AnalyRR},   result in
\begin{equation*}
i\lambda A^{-\frac{1}{2}}u_2= -\dfrac{\kappa_1}{\rho_1}A^\frac{1}{2}\varphi-\dfrac{\kappa_1}{\rho_1}A^{-\frac{1}{2}}\psi_x+\dfrac{\jmath}{\rho_1}A^{-\frac{1}{2}}(y-\varphi)-A^{\beta_1-\frac{1}{2}}u+A^\frac{1}{2}u_1,
\end{equation*}
 then,  as $\|A^{-\frac{1}{2}}\psi_x\|^2=\dual{-A^{-\frac{1}{2}}\psi_{xx}}{A^{-\frac{1}{2}}\psi}=\dual{A^\frac{1}{2}\psi}{A^{-\frac{1}{2}}\psi}=\|\psi\|^2\leq C\|A^\frac{1}{2}\psi\|^2$, $-\frac{1}{2}<0$ and  $\beta_1-\frac{1}{2}\leq \frac{\beta_1}{2}$, taking into account the continuous embedding $\mathfrak{D}(A^{\theta_2}) \hookrightarrow \mathfrak{D}(A^{\theta_1}),\;\theta_2>\theta_1$ and using \eqref{Eq112AnalyRR} and as $-1\leq -\frac{2\beta_1}{\beta_1+1}$,   result in 
\begin{eqnarray*}
\nonumber
|\lambda|^2\|A^{-\frac{1}{2}} u_2\|^2&\leq& C\{\|A^\frac{1}{2}\varphi\|^2+ \|A^\frac{1}{2}\psi\|^2+\|y-\varphi\|^2+\|A^\frac{\beta_1}{2}u\|^2\}+\|A^\frac{1}{2}u_1\|^2\\
\nonumber
&\leq &  C\{\|F\|_{\mathbb{H}_2}\|U\|_{\mathbb{H}_2}+|\lambda|^{-1}\|F\|^2_{\mathbb{H}_2}\}    \leq C|\lambda|^{-\frac{2\beta_1}{\beta_1+1}}\{ |\lambda|^\frac{2\beta_1}{\beta_1+1}\|F\|_{\mathbb{H}_2}\|U\|_{\mathbb{H}_2}+\|F\|^2_{\mathbb{H}_2} \}.
\end{eqnarray*}
Then
\begin{equation}
\label{Eq113AnalyRR}
\|A^{-\frac{1}{2}} u_2\|^2\leq  C|\lambda|^{-\frac{2(2\beta_1+1)}{\beta_1+1}}\big\{|\lambda|^\frac{2\beta_1}{\beta_1+1}\|F\|_{\mathbb{H}_2}\|U\|_{\mathbb{H}_2}+\|F\|^2_{\mathbb{H}_2} \big\}.
\end{equation}
On the  other hand, from $A^\frac{\beta_1}{2}u_2=A^\frac{\beta_1}{2}u-A^\frac{\beta_1}{2}u_1$,  \eqref{dis2-10} and  as $\mathfrak{D}(A^\frac{1}{2}) \hookrightarrow \mathfrak{D}(A^\frac{\beta_1}{2})$,  the inequality of \eqref{Eq112AnalyRR}, result in
\begin{equation}\label{Eq114AnalyRR}
\|A^\frac{\beta_1}{2} u_2\|^2 \leq  C\{ \|A^\frac{\beta_1}{2} u\|^2+\|A^\frac{\beta_1}{2}u_1\|^2
\} \leq  C|\lambda|^{-\frac{2\beta_1}{\beta_1+1}}\{|\lambda|^\frac{2\beta_1}{\beta_1+1}\|F\|_{\mathbb{H}_2}\|U\|_{\mathbb{H}_2}+\|F\|^2_{\mathbb{H}_2} \}.
\end{equation}
By Lions' interpolations inequality (Theorem \ref{Lions-Landau-Kolmogorov}), $0\in \big[-\frac{1}{2},\frac{\beta_1}{2}\big]$,  result in
\begin{equation}\label{Eq115AnalyRR}
 \|u_2\|^2\leq C(\|A^{-\frac{1}{2}}u_2\|^2)^\frac{\beta_1}{1+\beta_1}(\|A^\frac{\beta_1}{2}u_2\|^2)^\frac{1}{1+\beta_1}.
\end{equation}
Then, using \eqref{Eq113AnalyRR} and \eqref{Eq114AnalyRR} in \eqref{Eq115AnalyRR}, for $ |\lambda|>1$,  result in 
\begin{equation}\label{Eq118AnalyRR}
 \|u_2\|^2\leq C|\lambda|^{-\frac{4\beta_1}{1+\beta_1}}\{ |\lambda|^\frac{2\beta_1}{1+\beta_1}\|F\|_{\mathbb{H}_2}\|U\|_{\mathbb{H}_2}+\|F\|^2_{\mathbb{H}_2}\}.
\end{equation}
%{\color{red}
Therefore,   as $\|u\|^2\leq  \|u_1\|^2+ \|u_2\|^2$,  from  \eqref{Eq112AnalyRR},  \eqref{Eq118AnalyRR} and  as for $0\leq\beta_1\leq 1$ we have $|\lambda|^{-2}\leq |\lambda|^{-\frac{4\beta_1}{1+\beta_1}}$, result in 
\begin{equation}\label{Eq119AnalyRR}
 |\lambda|\|u\|^2\leq C_\delta|\lambda|^\frac{1-3\beta_1}{1+\beta_1}\{|\lambda|^\frac{2\beta_1}{1+\beta_1}\|F\|_{\mathbb{H}_2}\|U\|_{\mathbb{H}_2}+\|F\|^2_{\mathbb{H}_2}\}\qquad\rm{for}\qquad 0\leq\beta_1\leq 1.
\end{equation}
%%%%%%%%%%%%%%%%%%%%%%%%%%%%%%%%%%%%%%%%%%%%%%%%%%%%%%%%%%%
%%%%%%%%%%%%%%%%%%%%%%%%%%%%%%%%%%%%%%%%%%%%%%%%%%%%%%%%%%%
%%%%%%%%%%%  Estimativa de $|\lambda|\|s\|^2$    %%%%%%%%%%
%%%%%%%%%%%%%%%%%%%%%%%%%%%%%%%%%%%%%%%%%%%%%%%%%%%%%%%%%%%
%{\color{red}
{\bf On the other hand,   let's now estimate the missing term  $|\lambda|^\frac{2\beta_2}{1+\beta_2}\|s\|^2$: } It is assumed that $|\lambda|>1$.  Set $s=s_1+s_2$, where $s_1\in \mathfrak{D}(A)$ and $s_2\in \mathfrak{D}(A^0)$, with 
\begin{multline}\label{Eq110AnalyRRW}
i\lambda s_1+As_1=f^6 \qquad {\rm and}\qquad i\lambda s_2=-\dfrac{\kappa_2}{\rho_3}Ay -\dfrac{\kappa_2}{\rho_3}z_x-\dfrac{\jmath}{\rho_3}(y-\varphi)-\dfrac{\gamma_2}{\rho_3}A^{\beta_2}s+As_1.
\end{multline} 
Firstly,  applying in the product duality  the first equation in \eqref{Eq110AnalyRRW} by $s_1$, then  by $As_1$    and recalling that the operator $A$  is
self-adjoint, resulting in 
\begin{equation}\label{Eq112AnalyRRW}
|\lambda|\|s_1\|+|\lambda|^\frac{1}{2}\|A^\frac{1}{2}s_1\|+\|As_1\|\leq C\|F\|_{\mathbb{H}_2}. 
\end{equation}
%%%%%%%%
Applying the operator $A^{-\frac{1}{2}}$  in second equation of \eqref{Eq110AnalyRRW}, we have  
\begin{equation*}
i\lambda A^{-\frac{1}{2}}s_2= -\dfrac{\kappa_2}{\rho_3}A^\frac{1}{2}y -\dfrac{\kappa_2}{\rho_3}A^{-\frac{1}{2}}z_x-\dfrac{\jmath}{\rho_3}A^{-\frac{1}{2}}(y-\varphi)-\dfrac{\gamma_2}{\rho_3}A^{\beta_2-\frac{1}{2}}s+A^\frac{1}{2}s_1,
\end{equation*}
 then,  as $\|A^{-\frac{1}{2}}z_x\|^2=\|z\|^2\leq C\|A^\frac{1}{2}z\|^2$,  $0<\frac{1}{2}$ and  $\beta_2-\frac{1}{2}\leq \frac{\beta_2}{2}$,  taking into account the continuous embedding $\mathfrak{D}(A^{\theta_2}) \hookrightarrow \mathfrak{D}(A^{\theta_1}),\;\theta_2>\theta_1$,  lead to
\begin{eqnarray*}
|\lambda|^2\|A^{-\frac{1}{2}} s_2\|^2& \leq & C\{\|A^\frac{1}{2}y\|^2+\|A^\frac{1}{2}z\|^2+\|y-\varphi\|^2+\|A^\frac{\beta_2}{2}s\|^2\}+\|A^\frac{1}{2}s_1\|^2\\
&\leq & C\{\|F\|_{\mathbb{H}_2}\|U\|_{\mathbb{H}_2}+|\lambda|^{-1}\|F\|^2_{\mathbb{H}_2}\}    \leq C|\lambda|^{-\frac{2\beta_2}{\beta_2+1}}\{ |\lambda|^\frac{2\beta_2}{\beta_2+1}\|F\|_{\mathbb{H}_2}\|U\|_{\mathbb{H}_2}+\|F\|^2_{\mathbb{H}_2} \}.
\end{eqnarray*}
Then
\begin{equation}
\label{Eq113AnalyRRW}
\|A^{-\frac{1}{2}} s_2\|^2\leq  C|\lambda|^{-\frac{2(2\beta_2+1)}{\beta_2+1}}\big\{|\lambda|^\frac{2\beta_2}{\beta_2+1}\|F\|_{\mathbb{H}_2}\|U\|_{\mathbb{H}_2}+\|F\|^2_{\mathbb{H}_2} \big\}.
\end{equation}
On the  other hand,  from $A^\frac{\beta_2}{2}s_2=A^\frac{\beta_2}{2}s-A^\frac{\beta_2}{2}s_1$,  \eqref{dis2-10} and  as $\mathfrak{D}(A^\frac{1}{2}) \hookrightarrow \mathfrak{D}(A^\frac{\beta_1}{2})$,  the inequality of \eqref{Eq112AnalyRRW}, result in
\begin{equation}\label{Eq114AnalyRRW}
\|A^\frac{\beta_1}{2} s_2\|^2 \leq  C\{ \|A^\frac{\beta_2}{2} s\|^2+\|A^\frac{\beta_2}{2}s_1\|^2
\} \leq  C|\lambda|^{-\frac{2\beta_2}{\beta_2+1}}\{|\lambda|^\frac{2\beta_2}{\beta_2+1}\|F\|_{\mathbb{H}_2}\|U\|_{\mathbb{H}_2}+\|F\|^2_{\mathbb{H}_2} \}.
\end{equation}
By Lions' interpolations inequality $0\in \big[-\frac{1}{2},\frac{\beta_2}{2}\big]$,  result in
\begin{equation}\label{Eq115AnalyRRW}
 \|s_2\|^2\leq C(\|A^{-\frac{1}{2}}s_2\|^2)^\frac{\beta_2}{1+\beta_2}(\|A^\frac{\beta_2}{2}s_2\|^2)^\frac{1}{1+\beta_2}.
\end{equation}
Then, using  \eqref{Eq113AnalyRRW} and \eqref{Eq114AnalyRRW} in \eqref{Eq115AnalyRRW}, for $ |\lambda|>1$,  result in 
\begin{equation}\label{Eq118AnalyRRW}
 \|s_2\|^2\leq C|\lambda|^{-\frac{4\beta_2}{1+\beta_2}}\{ |\lambda|^\frac{2\beta_2}{1+\beta_2}\|F\|_{\mathbb{H}_2}\|U\|_{\mathbb{H}_2}+\|F\|^2_{\mathbb{H}_2}\}.
\end{equation}
Therefore,   as $\|s\|^2\leq  \|s_1\|^2+ \|s_2\|^2$,  from  \eqref{Eq112AnalyRRW},  \eqref{Eq118AnalyRRW} and  as for $0\leq\beta_2\leq 1$ we have $|\lambda|^{-2}\leq |\lambda|^{-\frac{4\beta_2}{1+\beta_2}}$, result in 
\begin{equation}\label{Eq119AnalyRRW}
 |\lambda|\|s\|^2\leq C_\delta|\lambda|^\frac{1-3\beta_2}{1+\beta_2}\{|\lambda|^\frac{2\beta_2}{1+\beta_2}\|F\|_{\mathbb{H}_2}\|U\|_{\mathbb{H}_2}+\|F\|^2_{\mathbb{H}_2}\}\qquad\rm{for}\qquad 0\leq\beta_2\leq 1.
\end{equation}
%%%%%%%%%%%%%%%%%%%%%%%%%%%%%%%%%%%%%%%%
{\bf  Finally,   let's now estimate the missing term  $|\lambda|^\frac{2\beta_3}{1+\beta_3}\|w\|^2$:}    It is assumed that $|\lambda|>1$.  Set $w=w_1+w_2$, where $w_1\in \mathfrak{D}(A)$ and $w_2\in \mathfrak{D}(A^0)$, with 
\begin{multline}\label{Eq110AnalyRRT}
i\lambda w_1+Aw_1=f^8 \qquad {\rm and}\qquad i\lambda w_2=-\dfrac{b_2}{\rho_4}Az +\dfrac{\kappa_2}{\rho_4}(y_x-z)-\dfrac{\gamma_3}{\rho_4}A^{\beta_3}w+Aw_1.
\end{multline} 
Firstly,  applying in the product duality  the first equation in \eqref{Eq110AnalyRRT} by $w_1$, then  by $Aw_1$    and recalling that the operator $A$  is
self-adjoint, resulting in 
\begin{equation}\label{Eq112AnalyRRT}
|\lambda|\|w_1\|+|\lambda|^\frac{1}{2}\|A^\frac{1}{2}w_1\|+\|Aw_1\|\leq C\|F\|_{\mathbb{H}_2}. 
\end{equation}
%%%%%%%%
Applying the operator $A^{-\frac{1}{2}}$  in second equation of \eqref{Eq110AnalyRRT}, we get  
\begin{equation*}
i\lambda A^{-\frac{1}{2}}w_2= -\dfrac{b_2}{\rho_4}A^\frac{1}{2}z +\dfrac{\kappa_2}{\rho_4}A^{-\frac{1}{2}}(y_x-z)-\dfrac{\gamma_3}{\rho_4}A^{\beta_3-\frac{1}{2}}w+A^\frac{1}{2}w_1,
\end{equation*}
 then, from $0\leq\beta_3\leq 1$, we have:     $\mathfrak{D}(A^\frac{\beta_3}{2}) \hookrightarrow \mathfrak{D}(A^{\beta_3-\frac{1}{2}})$ and $\mathfrak{D}(A^0) \hookrightarrow \mathfrak{D}(A^{-\frac{1}{2}})$, from estimates \eqref{EqvExponencial2} and \eqref{Eq112AnalyRRT},  lead to
\begin{equation}
\label{Eq113AnalyRRT}
\|A^{-\frac{1}{2}} w_2\|^2\leq  C|\lambda|^{-\frac{2(2\beta_3+1)}{\beta_3+1}}\big\{|\lambda|^\frac{2\beta_3}{\beta_3+1}\|F\|_{\mathbb{H}_2}\|U\|_{\mathbb{H}_2}+\|F\|^2_{\mathbb{H}_2} \big\}.
\end{equation}
On the  other hand, from $A^\frac{\beta_3}{2}w_2=A^\frac{\beta_3}{2}w-A^\frac{\beta_3}{2}w_1$,  \eqref{dis2-10} and  as $\mathfrak{D}(A^\frac{1}{2}) \hookrightarrow \mathfrak{D}(A^\frac{\beta_1}{2})$,  the inequality of \eqref{Eq112AnalyRRT}, result in
\begin{equation}\label{Eq114AnalyRRT}
\|A^\frac{\beta_1}{2} w_2\|^2 \leq  C\{ \|A^\frac{\beta_3}{2} w\|^2+\|A^\frac{\beta_3}{2}w_1\|^2
\} \leq  C|\lambda|^{-\frac{2\beta_3}{\beta_3+1}}\{|\lambda|^\frac{2\beta_3}{\beta_3+1}\|F\|_{\mathbb{H}_2}\|U\|_{\mathbb{H}_2}+\|F\|^2_{\mathbb{H}_2} \}.
\end{equation}
By Lions' interpolations inequality $0\in \big[-\frac{1}{2},\frac{\beta_3}{2}\big]$,  result in
\begin{equation}\label{Eq115AnalyRRT}
 \|w_2\|^2\leq C(\|A^{-\frac{1}{2}}w_2\|^2)^\frac{\beta_3}{1+\beta_3}(\|A^\frac{\beta_3}{2}w_2\|^2)^\frac{1}{1+\beta_3}.
\end{equation}
Then, using  \eqref{Eq113AnalyRRT} and \eqref{Eq114AnalyRRW} in \eqref{Eq115AnalyRRT}, for $ |\lambda|>1$,  result in 
\begin{equation}\label{Eq118AnalyRRT}
 \|w_2\|^2\leq C|\lambda|^{-\frac{4\beta_3}{1+\beta_3}}\{ |\lambda|^\frac{2\beta_3}{1+\beta_3}\|F\|_{\mathbb{H}_2}\|U\|_{\mathbb{H}_2}+\|F\|^2_{\mathbb{H}_2}\}.
\end{equation}
Therefore,   as $\|w\|^2\leq C\{ \|w_1\|^2+ \|w_2\|^2\}$,  from  \eqref{Eq112AnalyRRT},  \eqref{Eq118AnalyRRT} and  as for $0\leq\beta_3\leq 1$ we have $|\lambda|^{-2}\leq |\lambda|^{-\frac{4\beta_3}{1+\beta_3}}$, result in 
\begin{equation}\label{Eq119AnalyRRT}
 |\lambda|\|w\|^2\leq C_\delta|\lambda|^\frac{1-3\beta_3}{1+\beta_3}\{|\lambda|^\frac{2\beta_3}{1+\beta_3}\|F\|_{\mathbb{H}_2}\|U\|_{\mathbb{H}_2}+\|F\|^2_{\mathbb{H}_2}\}\qquad\rm{for}\qquad 0\leq\beta_3\leq 1.
\end{equation}
%%%%%%%%%%%%%%%%%%%%%%%%%%%%%%%%
On other hand, using \eqref{Eq119AnalyRR} in inequality \eqref{Eq04Lemma052}, we have
\begin{equation}\label{Eq01ParaVphixPsi}
|\lambda|\|\varphi_x-\psi\|^2\leq C_\delta|\lambda|^\frac{1-3\beta_1}{1+\beta_1}\{|\lambda|^\frac{2\beta_1}{1+\beta_1}\|F\|_{\mathbb{H}_2}\|U\|_{\mathbb{H}_2}+\|F\|^2_{\mathbb{H}_2}\}\quad\text{for}\quad 0\leq \beta_1\leq 1.
\end{equation}
Using \eqref{Eq119AnalyRRW} in inequality \eqref{Eq06Lemma052}, we have
\begin{equation}\label{Eq01ParaYxz}
|\lambda|\|y_x-z\|^2\leq C_\delta|\lambda|^\frac{1-3\beta_2}{1+\beta_2}\{|\lambda|^\frac{2\beta_2}{1+\beta_2}\|F\|_{\mathbb{H}_2}\|U\|_{\mathbb{H}_2}+\|F\|^2_{\mathbb{H}_2}\}\quad\text{for}\quad 0\leq \beta_2\leq 1.
\end{equation}

Now, using \eqref{Eq119AnalyRRT} in inequality \eqref{Ec-wZ}, we have
\begin{equation}\label{Eq01ParaA12z}
|\lambda|\|A^\frac{1}{2}z\|^2\leq C_\delta|\lambda|^\frac{1-3\beta_2}{1+\beta_2}\{|\lambda|^\frac{2\beta_2}{1+\beta_2}\|F\|_{\mathbb{H}_2}\|U\|_{\mathbb{H}_2}+\|F\|^2_{\mathbb{H}_2}\}\quad\text{for}\quad 0\leq \beta_2\leq 1.
\end{equation}

Furthermore,  taking  $
\phi:=\min\limits_{(\beta_1,\beta_2,\beta_3)\in (0,1)^3}\{\beta_1,\beta_2, \beta_3\}$ defined in \eqref{phi}, we have,  $0<\phi<1$ and from estimates; \eqref{Eq119AnalyRR},\eqref{Eq119AnalyRRW} and \eqref{Eq119AnalyRRT}, we obtain
\begin{equation}\label{Gevrey001}
|\lambda|^\frac{2\phi}{1+\phi} \{ \rho_1\|u\|^2+\rho_3\|w\|^2+\rho_4 \|s\|^2 \} \leq C_\delta\|F\|_{\mathbb{H}_2}\|U\|_{\mathbb{H}_2}\quad \text{for}\quad 0<(2\phi)/(\phi+1)<1.
\end{equation}
As, $0<\frac{2\phi}{1+\phi}<1$, from estimates: \eqref{Item01Lemma22}, \eqref{Item01Lemma52}, \eqref{EcLambdav} and \eqref{Item01Lemma82}, we get
\begin{equation}\label{Gevrey002}
|\lambda|^\frac{2\phi}{1+\phi}\{\jmath\|y-\varphi\|^2+\rho_5\|\theta\|^2+\rho_2\|v\|^2+b_1\|A^\frac{1}{2}\psi\|^2\}\leq C_\delta\|F\|_{\mathbb{H}_2}\|U\|_{\mathbb{H}_2}\quad \text{for}\quad 0<(2\phi)/(\phi+1)<1.
\end{equation}
Now, as $0<\frac{2\phi}{\phi+1}<1$, from estimates; \eqref{Eq01ParaVphixPsi},\eqref{Eq01ParaYxz} and \eqref{Eq01ParaA12z}, we obtain
\begin{equation}\label{Gevrey003}
|\lambda|^\frac{2\phi}{1+\phi}\{\kappa_1\|\varphi_x-\psi\|^2+\kappa_2\|y_x-z\|^2+b_2\|A^\frac{1}{2}z\|^2\}\leq C_\delta\|F\|_{\mathbb{H}_2}\|U\|_{\mathbb{H}_2}\quad \text{for}\quad 0<(2\phi)/(\phi+1)<1.
\end{equation}
Finally summing the estimates \eqref{Gevrey001},\eqref{Gevrey002} and  \eqref{Gevrey003},  we have
\begin{equation*}
|\lambda|^\frac{2\phi}{\phi+1}\|U\|_\mathcal{H}\leq C_\delta  \|F\|_\mathcal{H}\qquad \text{for}\qquad 0<(2\phi)/(\phi+1)<1. 
\end{equation*}
Therefore, 
the proof of this theorem is finished.
%%%%%%%%%%%%%%%%%%%%%%%%%%%%%%%%%%%%%%%%%%%%%%%%%
\end{proof}
%%%%%%%%%%%%%%%%%%%%%%%%%%%%%%%%%%%%%%%%%%%%%%%%%

%%%%%%%%%%%%%%%%%%%%%%%%%%%%%%

%%%%%%%%%%%%%%%%%%%%%%%%%%%%%%%%%%%%%%%%%%%%%%%%%%%
%%%%%%%%%%%%%%%%%%%%%%%%%%%%%%%%%%%%%%%%%%%%%%%%%%%
\end{document}